\documentclass[a4paper,11pt]{article}
\usepackage{latexsym}
\usepackage{amssymb}
\usepackage{enumerate}
\usepackage{lscape}
\usepackage[shortlabels]{enumitem}
\usepackage{xcolor}
\usepackage{commath}
\usepackage{comment}
\usepackage{amsfonts}
\usepackage{stmaryrd} 
\usepackage{amsmath,amsthm}
\usepackage{hyperref}
\usepackage{algorithm}
\usepackage{algorithmic}
\usepackage{relsize}
\usepackage[title]{appendix}
\usepackage{framed}
\usepackage{boites}
\usepackage[pdftex]{graphicx}
\newtheorem{mainthm}{Theorem}

\newtheorem{thm}{Theorem}

\newenvironment{@abssec}[1]{%
    \if@twocolumn

      \section*{#1}%
    \else

      \vspace{.05in}\footnotesize
      \parindent .2in
 {\upshape\bfseries #1. }\ignorespaces
    \fi}

    {\if@twocolumn\else\par\vspace{.1in}\fi}

\newenvironment{keywords}{\begin{@abssec}{\keywordsname}}{\end{@abssec}}

\newenvironment{AMS}{\begin{@abssec}{\AMSname}}{\end{@abssec}}

\newcommand\keywordsname{Key words}
\newcommand\AMSname{AMS subject classifications}
\newcommand\AMname{AMS subject classification}
\newcommand\restr[2]{{
\left.\kern-\nulldelimiterspace 
#1 
\vphantom{|} 
\right|_{#2} 
}}
\newtheorem{theorem}{Theorem}[section]
\newtheorem{lemma}[theorem]{Lemma}
\newtheorem{corollary}[theorem]{Corollary}
\newtheorem{proposition}[theorem]{Proposition}
\newtheorem{remark}[theorem]{Remark}

\newcommand{\NN}{\mathbb{N}}

\newcommand{\RR}{\mathbb{R}}

\renewcommand{\SS}{\mathbb{S}}

\newcommand{\tin}{{\text{in }}}
\newcommand{\ton}{{\text{on }}}

\def\XXint#1#2#3{{\setbox0=\hbox{$#1{#2#3}{\int}$}
\vcenter{\hbox{$#2#3$}}\kern-.5\wd0}}

\newcommand{\im}{\mathop{\mathrm{Im}}}   
  
\newcommand{\link}{\mathop{\circ\kern-.35em -}}

\newcommand{\codim}{\mathop{\mathrm{codim}}}

\renewcommand{\ker}  
{\mathop{\mathrm{Ker}}}

\newcommand{\ol}{\overline}
\newcommand{\pa}{\partial}

\newcommand{\dv}{\mathop{\mathrm{div}}}
\newcommand{\na}{\nabla}

\newcommand{\gr}{\nabla}

\newcommand{\al}{\alpha}
\newcommand{\be}{\beta}
\newcommand{\ga}{\gamma}

\newcommand{\De}{\Delta}
\newcommand{\ve}{\varepsilon}
 
\newcommand{\la}{\lambda}

\newcommand{\si}{\sigma}

\newcommand{\te}{\theta}

\newcommand{\om}{\omega}
\newcommand{\Om}{\Omega}
\newcommand{\rn}{{\mathbb{R}}^N}
\newcommand{\sg}{\sigma}

\newcommand{\id}{{\rm Id}}
\newcommand\setbld[2]{\left\{ #1 \;\middle |\; #2\right\}}
\newcommand\jump[1]{\left\llbracket #1\right\rrbracket}

\newcommand{\cdottone}{{\boldsymbol{\cdot}}}

\newcommand{\cA}{\mathcal{A}}
\newcommand{\cB}{\mathcal{B}}
\newcommand{\cC}{\mathcal{C}}
\newcommand{\cD}{\mathcal{D}}

\newcommand{\cM}{\mathcal{M}}

\newcommand{\cX}{\mathcal{X}}
\newcommand{\cY}{\mathcal{Y}}

\newcommand{\Orig}{O}
\newcommand{\identmatrix}
{{\mathbf{I}}}

\numberwithin{equation}{section}

\title{\bf Face 2-phase:
\\
how much overdetermination is enough to get symmetry in two-phase problems
}

\author{Lorenzo Cavallina \, Giorgio Poggesi \
}

\date{}

\begin{document}

\maketitle

\begin{abstract}
We provide a full characterization of multi-phase problems under a large class of overdetermined Serrin-type conditions. Our analysis includes both symmetry and asymmetry (including bifurcation) results. A broad range of techniques is needed to obtain a full characterization of all the cases, including applications of results obtained via the moving planes method, approaches via integral identities in the wake of Weinberger, applications of the Crandall--Rabinowitz theorem, and the Cauchy--Kovalevskaya theorem. The multi-phase setting entails intrinsic difficulties that make it difficult to predict whether a given overdetermination will lead to symmetry or asymmetry results; the results of our analysis are significant as they answer such a question providing a full characterization of both symmetry and asymmetry results.
\end{abstract}



\pagestyle{plain}
\thispagestyle{plain}

\section{Introduction and main results}\label{sec:introduction}

We would like to start by giving the following alternative take on the celebrated theorem by J. Serrin \cite{Se1971}. Let $\Om$ be a bounded domain (that is, a bounded connected open set) of $\RR^N$ ($N\ge 2$) whose boundary is made of regular points for the Dirichlet Laplacian (see for instance \cite[Chapter 8]{GT} where the well-known Wiener criterion is also discussed). Consider the solution $u$ to the following boundary value problem.
\begin{equation}\label{eq torsion pb}
  \De u = N\quad \text{in }\Om,\quad 
  u=0\quad \text{on }\pa\Om.
\end{equation}
Let $\om\subset\subset\Om$ be a subdomain\footnote{We say that $A\subset\subset B$ if $\ol A\subset B$.}. We say that $\pa\om$ is an \emph{overdetermined level set} for the function $u$ if both $u$ and $|\gr u|$ are constant functions on $\pa\om$, that is
\begin{equation*}
    u\equiv a, \quad | \na u | \equiv c \quad \text{on }\pa\om,
\end{equation*}
for some real constants $a$ and $c$. Notice that we must have
\begin{equation}\label{quella senza integrazione per parti}
   c >0;
\end{equation}
in fact, assuming that $c=0$, the subharmonicity of $|\na u|^2$ would give that $| \na u| \equiv 0$ in $\om$, which contradicts $\Delta u = N$ in $\om$.

The following proposition can be obtained as a corollary of Serrin's result \cite{Se1971}. 
\begin{proposition}\label{one phase proposition}
Let $\Om\subset\rn$ be a bounded domain whose boundary is made of regular points for the Dirichlet Laplacian. Then, the following are equivalent:
\begin{enumerate}[(i)]
    \item $\Om$ is a ball.
    \item The solution of \eqref{eq torsion pb} admits an overdetermined level set $\pa\om$.
\end{enumerate}
 \end{proposition}
 \begin{proof}
  If $\Om$ is a ball, then $u$ is radial, and thus every level set is overdetermined. Let us now consider the reverse implication. Notice that the overdetermined level set $\pa\om$ is automatically smooth by standard interior regularity for elliptic equations and the fact that $|\gr u|>0$ on $\pa\om$ by \eqref{quella senza integrazione per parti}. The symmetry result of \cite{Se1971} applied to the domain $\om$ yields that $\om$ must be a ball and $u$ must be radial inside $\om$. Now, since $u$ is real analytic inside the whole $\Om$, then $u$ must be radial in $\Om$ as well. Finally, since $\pa\Om$ is made of regular points, the solution $u$ is continuous up to the boundary by \cite[Theorem 8.30]{GT}. Thus, $\pa \Om$ coincides with the zero level set of $u$, and hence $\Om$ is a ball, as claimed.
 \end{proof}

This paper aims to study how the result of Proposition~\ref{one phase proposition} generalizes to the following two-phase setting. That is, let $(D,\Om)$ be a pair of bounded domains of $\RR^N$ ($N \ge 2$) with $D \subset\subset \Om$ such that $\Om\setminus \ol D$ is connected, and consider the boundary value problem  
\begin{equation}\label{eq:Dirichlet problem}
    \begin{cases}
        \dv(\sg \gr u)=N\quad \text{in }\Om,\\
        u=0\quad \text{on }\pa\Om,
    \end{cases}
\end{equation}
where $\sg$ is the piece-wise constant function defined by 
\begin{equation*}
    \sg:=\sg_c \cX_D+\cX_{\Om\setminus D}
\end{equation*}
for some given positive constant $\sg_c>0$.

We say that a function $u\in H_0^1(\Om)$ is a solution of \eqref{eq:Dirichlet problem} if it satisfies
\begin{equation}\label{eq: weak form}
    -\int_\Om \langle \sg \gr u, \gr \varphi\rangle dx=N\int_\Om \varphi \, dx \quad \text{for all }\varphi\in H_0^1(\Om). 
\end{equation}
We recall that (see \cite{AthaStra, LiVogelius, XB2013, Zhuge}), if $\pa D$ is sufficiently smooth (say, of class $C^{1,\al}$), 
then the solution $u$ satisfies the following transmission problem:
\begin{equation}\label{eq: transmission problem}
  \begin{cases}
      \sg_c \De u = N \quad \text{in } D,\\
      \De u = N \quad \text{in } \Om\setminus \ol D,\\
      \jump{u}=0\quad \text{on }\pa D,\\
      \jump{\sg u_\nu}=0\quad \text{on }\pa D,\\
      u=0\quad \text{on }\pa\Om.
  \end{cases}  
\end{equation}
Here, the quantity $\jump{\cdottone}$, called the \emph{jump} through the interface $\pa D$ is defined as follows: for any function $f\in H^1(\Om)$, we set 
\begin{equation*}
    \jump{f}:=\restr{f}{\pa^+ D}- \restr{f}{\pa^- D},
\end{equation*}
where $\restr{f}{\pa^+ D}$ and $\restr{f}{\pa^- D}$ denote the traces of $f$ on $\pa D$ taken from $\Om\setminus \ol D$ and $D$ respectively.
Moreover, we note that the normal derivative $u_\nu$ in the above (on both sides of $\pa D$) is to be considered with respect to the outer unit normal $\nu$ to $\pa D$.
The jump conditions on $\pa D$ in \eqref{eq: transmission problem} are usually referred to as \emph{transmission conditions}.

For a pair of nonnegative integers\footnote{Here and throughout the paper, $\NN$ will denote the set of positive integers.} $a,b\in \NN\cup\{0\}$, we say that the solution to \eqref{eq:Dirichlet problem} satisfies an overdetermination of type $(a,b)$ if there exist domains $\{\om_i\}_{i=1}^{a+b}$ satisfying
\begin{equation}\label{om_i}
    \om_1\subset\subset  \dots \subset\subset \om_a\subset\subset D \subset \subset \om_{a+1} \subset\subset\dots \subset\subset \om_{a+b} \subset\subset \Om, 
\end{equation}
such that, for each $i=1,\dots, a+b$, the boundary $\pa\om_i$ is an overdetermined level set for the solution to \eqref{eq:Dirichlet problem} (see Figure \ref{fig:enter-label}).
That is, we have
\begin{equation}\label{a_i c_i}
    u\equiv a_i, \quad |\na u| \equiv c_i \quad \text{on }\pa\om_i, \quad i=1,\dots, a+b,
\end{equation}
for some real constants $a_i$ and $c_i$. Notice that we must have
\begin{equation}\label{quell'altra senza integrazione per parti}
    c_i >0, \quad \text{ for any } i=1,\dots, a+b,
\end{equation}
which can be proved as follows. Consider a ball $B \subset \om_i$ with outward unit normal $\nu_{\pa B}$ and a point $y_0 \in \pa \om_i$  such that $\pa B \cap \pa D = \varnothing$ and $y_0 \in \pa B \cap \pa \om_i$.\footnote{For instance, take any $x_0\in \om$ such that $\mathrm{dist}(x_0,\pa\om) < \mathrm{dist}(x_0,\pa D)$ (this condition is necessary if $D \subset \om_i$, whereas it is trivially satisfied if $\om_i \subset D$) and consider $B:=B_{\mathrm{dist}(x_0,\pa\om_i)}(x_0)$ and $y_0 \in \pa B \cap \pa \om_i$.}
Noting that $u \equiv u(y_0)= a_i$ on $\pa\om_i$, the maximum principle (Lemma \ref{weak maximum pple}) gives that $u\le u(y_0)$ in $\om$. An application of the Hopf lemma in $B$ thus gives that
$$c_i 
= |\na u (y_0)| \ge u_{\nu_{\pa B}} (y_0) > 0,$$
and hence \eqref{quell'altra senza integrazione per parti}.

Moreover, to simplify matters, throughout this paper, we will assume that each $\pa \om_i$ is connected.
\begin{figure}
    \centering
\includegraphics[width=0.5\linewidth]{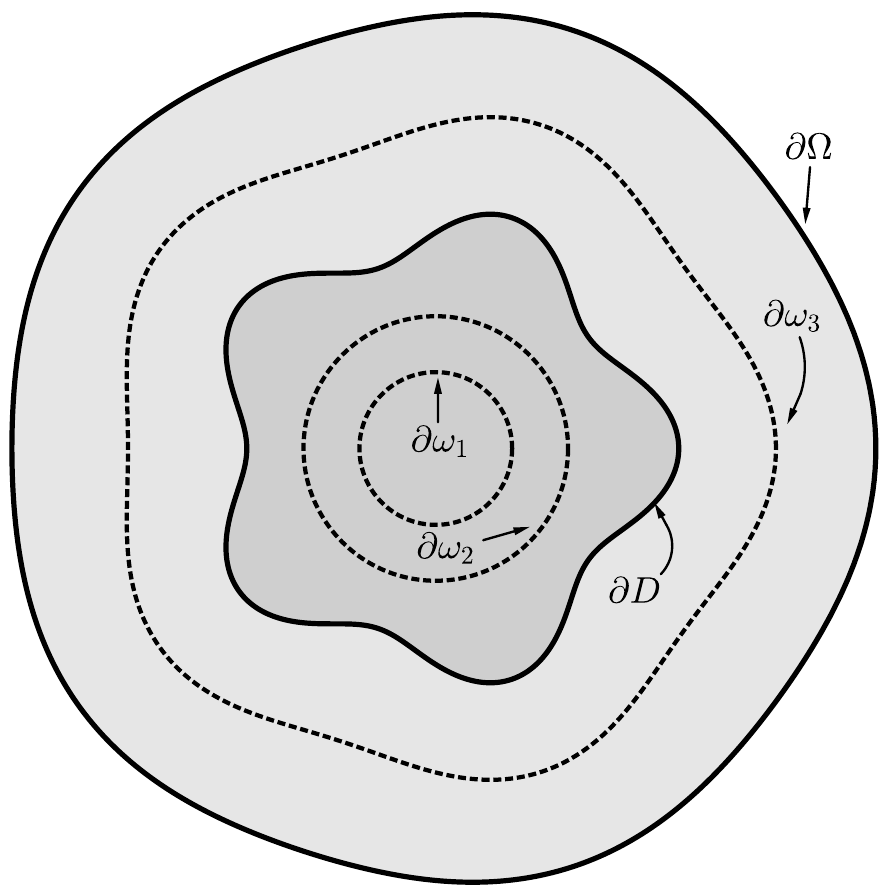}
    \caption{Problem setting when $a=2$, $b=1$. Also compare with Theorem~\ref{(1,1) bifurcation}}
    \label{fig:enter-label}
\end{figure}

In what follows, we present the main results of this paper, which provide a full characterization of the above-mentioned overdetermined problem. 
The relationship between the main results of this paper can be summarized in Figure~\ref{table}.

\begin{mainthm}\label{(0,2) symmetry}
Assume $b\in \NN$, $b\ge 2$, and let $\Om$ be a bounded domain whose boundary is made of regular points for the Dirichlet Laplacian. Let $D\subset\subset\Om$ be a bounded domain whose boundary is of class\footnote{As noticed in Remark \ref{rem:improvement of ThmI using B}, thanks to \cite{bessatsata}, Theorem \ref{(0,2) symmetry} remains valid even without assuming the $C^2$ regularity of $\partial D$.} $C^2$. Moreover, assume that $\Om\setminus\ol D$ is connected.
Then, $(D,\Om)$ satisfies some overdetermination of type $(0,b)$ if and only if $(D,\Om)$ are concentric balls. 
\end{mainthm}
Theorem~\ref{(0,2) symmetry} is obtained by combining a symmetry result in annular domains due to Sirakov \cite{Sirakov} and a symmetry result in the two-phase setting due to Sakaguchi \cite{Sak bessatsu}. We remark that this theorem is sharp in the sense that there exist counterexamples to radial symmetry if $b\le 1$ (see Theorem~\ref{(1,0) asymmetry} below). We mention that this theorem is also sharp with respect to the number of layers, in the sense discussed in Remark \ref{remark bello}.   
\begin{mainthm}\label{(1,1) D=ball symmetry}
Let $a\in\NN$, let $D_0$ be a ball, and let $\Om\supset\supset D_0$ 
 be a bounded domain whose boundary is made of regular points for the Dirichlet Laplacian.  
Then, the pair $(D_0,\Om)$ satisfies some overdetermination of type $(a,1)$ if and only if $(D_0,\Om)$ are concentric balls. 
\end{mainthm}
The proof of Theorem~\ref{(1,1) D=ball symmetry} requires more work and relies on the use of integral identities in the wake of Weinberger \cite{Weinberger} (see also \cite{Payne-Schaefer, NT2018, MP2, DPV}) while exploiting the new setting in an innovative way. More precisely, in Lemma~\ref{lem:fundamental identity having u explicit in D} we obtain a fundamental integral identity which provides a general necessary and sufficient condition for overdetermination of type $(1,0)$ (see also Theorem~\ref{thm:symmetry for internal only}) for general $D$. As a result, Theorem~\ref{(1,1) D=ball symmetry} is obtained by exploiting the additional assumptions. Two alternative proofs of Theorem~\ref{(1,1) D=ball symmetry} are provided in section \ref{sec:particular cases}.

We stress that Theorem~\ref{(1,1) D=ball symmetry} is sharp, in the sense that if any of the assumptions 
\begin{enumerate}[$(i)$]
\item $D=$ ball,
\item external overdetermination,
\item internal overdetermination
\end{enumerate}
is removed, then counterexamples to symmetry can be obtained. In fact, Theorem~\ref{(1,1) bifurcation} below provides the desired counterexample in the case where $(i)$ is removed, whereas, the counterexample in the case where $(ii)$ is removed is provided by Theorem~\ref{(1,0) asymmetry}.
Finally, \cite{CY2020} provides the desired counterexample in the case where $(iii)$ is removed.

\begin{mainthm}\label{(1,1) bifurcation}
Let $a\in\NN$.
Then, there exist pairs of bounded domains $(D,\Om)$, with analytic boundaries, that are not concentric balls but satisfy some overdetermination of type $(a,1)$. 
\end{mainthm}
The proof of Theorem~\ref{(1,1) bifurcation} relies on an application of the Crandall-Rabinowitz bifurcation Theorem~\cite{CR} to suitable shape ``functionals", made possible by the use of shape calculus (see for instance \cite{HP2005}). As a crucial tool, in Lemma~\ref{lem bifurcation}, we provide some new unified machinery to show the existence of bifurcation solutions to general overdeteremined problems in annular domains, which is of independent interest.

\begin{mainthm}\label{(1,0) asymmetry}
Let $a\in\NN$ and let $D_0$ be a ball. 
Then, there exists some domain $\Om\supset\supset D_0$, with analytic boundary, such that $(D_0,\Om)$ are not concentric balls but satisfy some overdetermination of type $(a,0)$. 
\end{mainthm}
Theorem~\ref{(1,0) asymmetry} is shown by constructing an explicit counterexample that exploits a quantitative version of the celebrated Cauchy--Kovalevskaya theorem~\cite{walter}.

\begin{figure}[h]
    \centering    \includegraphics[width=0.8\linewidth]{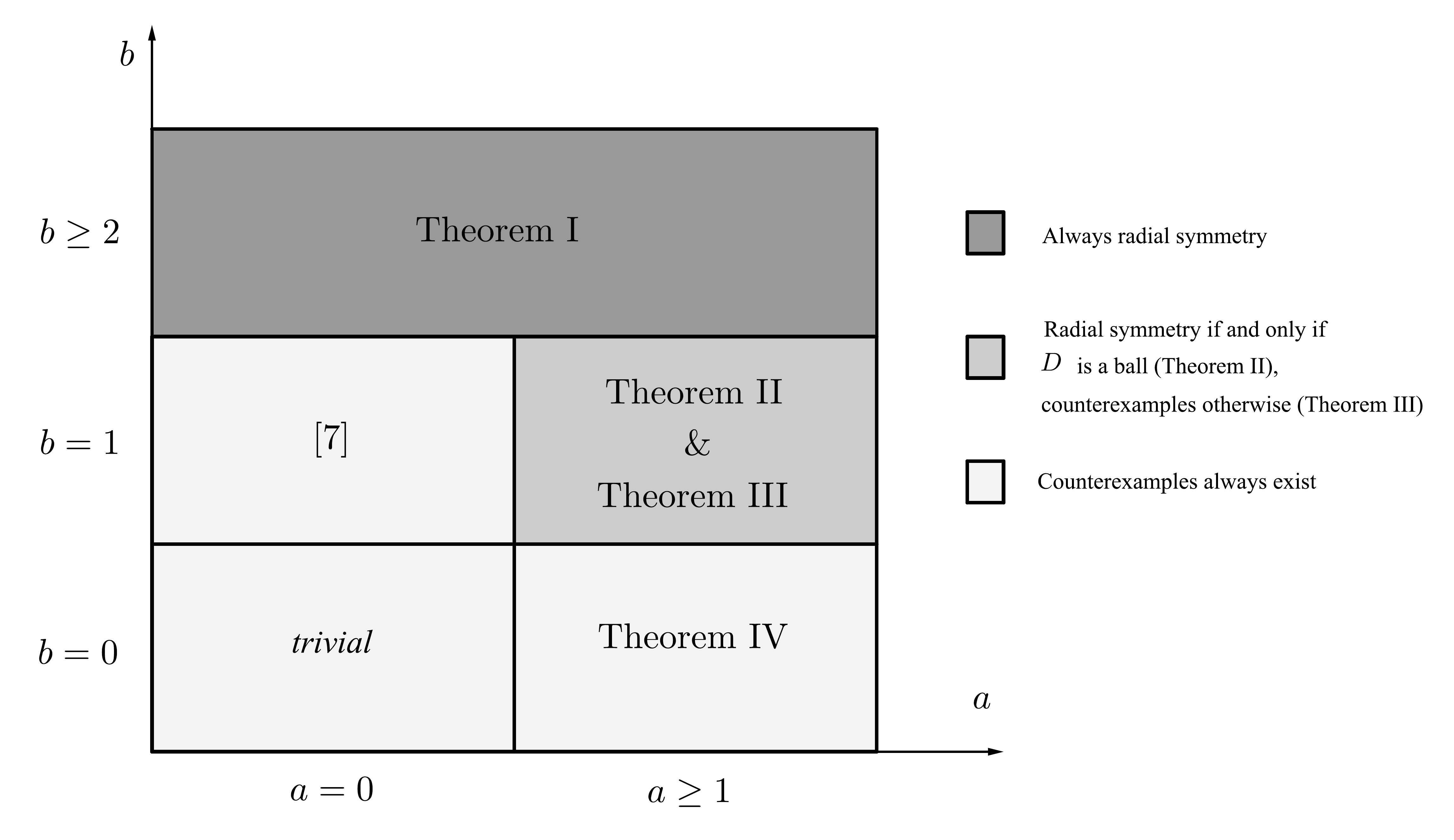}
    \caption{Concerning the existence of non radially symmetric configurations $(D,\Om)$ that satisfy overdetermination of type $(a,b)$ for general $a,b\in \NN\cup\{0\}$.}
    \label{table}
\end{figure}

This paper is organized as follows. 
In section \ref{sec:preliminary} we discuss how the various types of overdetermination are related for different values of $a$ and $b$. As a result, it will be enough to study settings where at most two distinct overdetermined level sets are present. In section \ref{sec:(02)} we provide a short proof of Theorem~\ref{(0,2) symmetry} by combining the known symmetry results of Sirakov \cite{Sirakov}  and Sakaguchi \cite{Sak bessatsu}. In section \ref{sec:particular cases}, we provide (see Theorem~\ref{thm:symmetry for internal only}) a general necessary and sufficient condition for the symmetry under overdetermination of type $(1,0)$ and give two alternative proofs for the symmetry result Theorem~\ref{(1,1) D=ball symmetry}. Section \ref{sec:counter(1,1)} is devoted to the proof of Theorem~\ref{(1,1) bifurcation}, where we construct a non-radial solution by means of the Crandall--Rabinowitz theorem. Finally, in section \ref{sec:(a,0)} we prove Theorem~\ref{(1,0) asymmetry} by constructing a non-radial solution via the Cauchy--Kovalevskaya theorem.

\section{Some preliminary simplifications}\label{sec:preliminary}
\subsection{On the case when either $\pa\om_a=\pa D$ or $\pa\om_{a+b}=\pa \Om$}
For simplicity, in the introduction, we limited our attention to the case where $\om_a\subset\subset D$ and $\om_{a+b}\subset\subset \Om$. 
In this subsection, we consider the neglected cases $\om_a=D$ and $\om_{a+b}=\Om$. In particular, we show that the former is a very strong constraint, equivalent to $(D,\Om)$ being concentric balls, while the case $\om_{a+b}\subset\subset\Om$ of the introduction can be easily reduced to the latter.
\begin{proposition}\label{prop:special case pa D level set}
Let $\pa\om_a=\pa D$ be an overdetermined level set and let $\pa\Om$ be made of regular points for the Dirichlet Laplacian. Then $(D, \Om)$ are concentric balls. 
\end{proposition}
\begin{proof}
Let $u$ be the solution to \eqref{eq:Dirichlet problem} in $(D,\Om)$ and consider the following auxiliary function:
\begin{equation*}
    v:=\begin{cases}
        \sg_c\big( u-\restr{u}{\pa D}\big)+\restr{u}{\pa D} \quad \tin D,\\
        \quad u\quad \tin \Om\setminus D.
    \end{cases}
\end{equation*}
The function $v$ then solves \eqref{eq torsion pb} and has $\pa\om_a=\pa D$ as an overdetermined level set. Thus, by Proposition~\ref{one phase proposition}, $v$ is radial, and $\om_a=D$ and $\Omega$ are concentric balls.  
\end{proof}
\begin{remark}\label{rem:special case pa D level set}
Actually, a more general result holds. Indeed, by employing the same auxiliary function $v$ and Serrin's result, we see that spherical symmetry follows under the broader assumption that $\pa D$ is (contained in) a (non necessarily overdetermined) level set and that there exists some overdetermined level set either completely contained inside $\ol\Om\setminus D$ or completely contained inside $\ol D$. 
\end{remark}

For $a\in \NN\cup\{0\}$ and $b\in \NN$, we say that $(D,\Om)$ satisfies an overdetermination of type $(a,b)^\star$ if it satisfies an overdetermination of type $(a,b-1)$ and $\pa\Om=\pa\om_{a+b}$ is a smooth\footnote{To simplify matters we shall assume that $\pa\om_{a+b}$ is of class $C^{2,\al}$, whence, recalling \eqref{quell'altra senza integrazione per parti}, \cite{Kinderlehrer Nirenberg} implies that $\pa\om_{a+b}$ is an analytic surface. We mention that the $C^{2,\al}$ assumption may be dropped provided that the boundary conditions on $\pa\om_{a+b}$ are intended in a suitable weak sense (see \cite{Vo}). On a related note, we stress that, whenever $\om_i\subset\subset\Om$, in light of the interior regularity of the solution $u$ to \eqref{eq:Dirichlet problem} and \eqref{quell'altra senza integrazione per parti}, $\pa\om_i$ is an analytic surface.} overdetermined level set. 
In the sense of the following two lemmas, the study of overdetermination of type $(a,b)$ can be reduced to that of overdetermination of type $(a,b)^\star$. 
\begin{lemma}\label{star implies nonstar for symmetry}
Let $a\in \NN\cup\{0\}$ and $b\in \NN$. Then, $(i)\implies(ii)$.
\begin{enumerate}[(i)]
    \item If $(D,\Om)$ satisfy an overdetermination of type $(a,b)^\star$, then they are concentric balls. 
    \item If $(D,\Om)$ satisfy an overdetermination of type $(a,b)$ and $\pa\Om$ is made of regular points for the Dirichlet Laplacian, then $(D,\Om)$ are concentric balls. 
\end{enumerate}
\end{lemma}
\begin{proof}
Suppose that $(i)$ holds and that the pair $(D,\Om)$ satisfies an overdetermination of type $(a,b)$ with $\pa\Om$ made of regular points for the Dirichlet Laplacian. In other words, $(D,\om_{a+b})$ satisfies an overdetermination of type $(a,b)^\star$. Thus, by $(i)$, $(D,\om_{a+b})$ are concentric balls and the solution $u$ of \eqref{eq:Dirichlet problem} is radial in $\om_{a+b}$. Finally, since $u$ is real analytic in $\Om\setminus \ol D$, $u$ is radial up to $\ol\Om$. In particular, since $\pa\Om$ is a level set of $u$, $\pa\Om$ is a sphere concentric with $D$, proving $(ii)$.
\end{proof}
\begin{lemma}\label{star implies nonstar for asymmetry}
Let $a\in \NN\cup\{0\}$, $b\in \NN$ and let $D$ be a bounded domain. Then, $(i)\implies (ii)$.
\begin{enumerate}[(i)]
    \item There exists some domain $\Om\supset\supset D$ such that $(D,\Om)$ are not concentric balls but they satisfy an overdetermination of type $(a,b)^\star$. 
    \item There exists some domain $\widetilde\Om\supset\supset D$ such that $(D,\widetilde\Om)$ are not concentric balls but they satisfy an overdetermination of type $(a,b)$. 
\end{enumerate}
\end{lemma}
\begin{proof}
Suppose $(i)$ holds, that is, there exists some bounded domain $\Om$ such that $(D,\Om)$ are not concentric balls but satisfy an overdetermination of type $(a,b)^\star$. By the local regularity result \cite[Theorem 2]{Kinderlehrer Nirenberg}, the overdetermined level set $\pa\Om$ is an analytic surface. Thus, one can apply the Cauchy--Kovalevskaya theorem to construct an extension $\widetilde u$ of the solution $u$ to problem \eqref{eq:Dirichlet problem} such that $\De \widetilde u= N$ in a small neighborhood $U$ of $\pa\Om$ by imposing the following Cauchy data:
\begin{equation*}
\widetilde u = \restr{u}{\pa\Om} \equiv 0,  \quad \widetilde{u}_\nu = u_\nu\equiv {\rm const.}>0 \quad \ton \pa\Om.      
\end{equation*}
Since, by construction $\widetilde u \equiv 0$ and $\widetilde{u}_\nu \equiv {\rm const.}>0$ on $\pa\Om$, for all sufficiently small $\ve>0$, one can find a larger domain $\widetilde\Om\supset\supset\Om$ with $\pa \widetilde\Om\subset U$, such that $\widetilde{u}=\varepsilon$ on $\pa\widetilde\Om$ (the interested reader is invited to compare this to the similar construction performed in section \ref{sec:(a,0)}, under less straightforward assumptions).
Recall that, by assumption, $(D,\Om)$ are not concentric balls. If $D$ is not a ball, then, in particular, $(D,\widetilde\Om)$ are also not concentric balls. On the other hand, if $D$ is a ball and $\Om$ is not a ``ball concentric with $D$" (by this we mean that $\Om$ may or may not be a ball, but if it is, its center must be distinct from that of $D$), by the arbitrariness of $\ve>0$, we can choose some $\ve>0$ so that the enlarged domain $\widetilde\Om$ is also not a ``ball concentric with $D$".
It is now immediate to notice that the function $\widetilde u -\ve$ is a solution of \eqref{eq:Dirichlet problem} in $(D,\widetilde\Om)$, which satisfies an overdetermination of type $(a,b)$ with $\om_{a+b}:=\Om$, proving $(ii)$.
\end{proof}

\subsection{On the type of overdetermination}
For $a_1,b_1,a_2,b_2\in\NN\cup\{0\}$, we say that $(a_1,b_2)\le (a_2,b_2)$ if and only if $a_1\le a_2$ and $b_1\le b_2$. 

Clearly, if $(a_1,b_2)\le (a_2,b_2)$ holds, an overdetermination of type $(a_2,b_2)$ is ``stronger" than one of type $(a_1,b_1)$. 
In particular, if for some $(a_1,b_1)$ we manage to show that overdetermination of type $(a_1,b_1)$ implies spherical symmetry, then the same conclusion must hold for any $(a_2,b_2)$ with $(a_1,b_1)\le (a_2,b_2)$. 
On the other hand, if there exists a pair $(D,\Om)$ that is not given by concentric balls but satisfies an overdetermination of type $(a_1,b_1)$, then the same pair $(D,\Om)$ trivially satisfies overdetermination of type $(a_2,b_2)$ for all $(a_2,b_2)\le (a_1,b_1)$. 

On a related note, we remark that any pair $(D,\Om)$ that satisfies an overdetermination of type $(1,b)$ must also satisfy an overdetermination of type $(a,b)$ for all $a\in \NN$. Indeed, if $u$ denotes the solution to \eqref{eq:Dirichlet problem} in $(D,\Om)$, then Serrin's result \cite{Se1971} applied to $\om_1$ yields that $\om_1$ is a ball and $\restr{u}{\om_1}$ is radial with respect to the center of $\om_1$. As a consequence, all level sets of $u$ that lie inside $\om_1$ are overdetermined. In other words, $(D,\Om)$ satisfies an overdetermination of type $(a,b)$ for all $a\in \NN$.

To conclude, we remark that the case of overdetermination of type $(0,0)$ (that is, no overdetermination) is trivial, while it is known that overdetermination of type $(0,1)$ is not enough to obtain spherical symmetry even under the assumption that $D$ is a ball (this follows from the bifurcation analysis done in \cite{CY2020}, where an overdetermination of type $(0,1)^\star$ is considered).
Thus, by the discussion above, our analysis is simplified to such an extent that, in order to show Theorems \ref{(1,1) D=ball symmetry} and \ref{(1,1) bifurcation} it is enough to consider overdetermination of type $(1,1)^\star$.


\section{Symmetry results for overdetermination of type \texorpdfstring{$(0,2)$}{(0,2)}}\label{sec:(02)}
We start by providing weak and strong maximum-type principles in the two-phase setting.

\begin{lemma}[Weak maximum principle]\label{weak maximum pple}
Let $u$ solve \eqref{eq:Dirichlet problem}. Then, $u\le 0$ in $\Om$.
\end{lemma}
\begin{proof}
    Take $\varphi:= \max\{u,0\} \in H^1_0(\Om)$, the positive part of $u$, as a test function. Then
    \begin{equation*}
        0\le \int_\Om \langle\sg \gr u , \gr \varphi\rangle \, dx = -N\int_\Om \varphi\, dx\le 0. 
    \end{equation*}
As a result, $\int_\Om \varphi\, dx=0$, which in turn implies $\varphi\equiv 0$ in $\Om$. This concludes the proof.     
\end{proof}

\begin{lemma}[Strong maximum principle]\label{strong maximum pple}
    Let $u$ solve \eqref{eq:Dirichlet problem}. If $\pa D$ is of class $C^{1,\al}$, then $u<0$ in $\Om$.
\end{lemma}
\begin{proof}
We will show that $\restr{u}{\pa D}<0$: the conclusion then follows by the maximum principle in $D$ and $\Om\setminus\ol D$. We recall that since $\pa D$ is of class $C^{1,\al}$, then $u$ is of class $C^{1,\al}$ in both $\ol D$ and $\Om\setminus D$ and it satisfies \eqref{eq: transmission problem}.
Let $u^-:=\restr{u}{\ol D}$, $u^+:=\restr{u}{\ol \Om\setminus D}$. Notice that, $\restr{u}{\pa D}\le 0$ by Lemma~\ref{weak maximum pple}. The functions $u^-$ and $u^+$ solve: 

\begin{minipage}{0.5\textwidth}
  \begin{equation*}
      \begin{cases}
          \sg_c \De u^-=N\quad \tin D,\\
          u^-=\restr{u}{\pa D}\le 0 \quad \ton \pa D.
      \end{cases}
  \end{equation*}
  \end{minipage}
\begin{minipage}{0.5\textwidth}
  \begin{equation*}
      \begin{cases}
         \De u^+=N\quad \tin \Om\setminus \ol D,\\
          u^+=\restr{u}{\pa D}\le 0 \quad \ton \pa D,\\
          u^+=0\quad \ton \pa\Om.
      \end{cases}
  \end{equation*}
  \end{minipage}
Now, assume by contradiction that $\displaystyle\max_{\pa D}u=u(x_0)=0$ for some $x_0\in \pa D$. Then, by the Hopf lemma (see for instance \cite{ABMMZ}) at $x_0$ for $u^-$ and $u^+$, we have $u_\nu^-(x_0)>0$ and $u_\nu^+(x_0)<0$. On the other hand, the transmission condition $\jump{\sg u_\nu}=0$ on $\pa D$ yields 
  \begin{equation*}
      0<\sg_c u_\nu^-(x_0)=u_\nu^+(x_0)<0,
  \end{equation*}
a contradiction. 
\end{proof}

We now provide the desired symmetry result for overdetermination of type $(0,2)$.
\begin{proof}[Proof of Theorem~\ref{(0,2) symmetry}]
Under the notation given by \eqref{om_i} and \eqref{a_i c_i}, Lemma~\ref{strong maximum pple} in $\om_2$ (applied to $u-a_2$) gives that $u<a_2$ in $\om_2$ and hence $a_1 <a_2$. Moreover, Lemma~\ref{strong maximum pple} in $\om_1$ gives that $u < a_1 = u|_{\pa\om_1} $ in $\om_1$, and hence that $c_1 = u_\nu \ge 0$ on $\pa\om_1$. Thus, \cite[Theorem 2]{Sirakov}, which is based on the moving planes method, applies (to the function $a_1-u$ in $\om_2\setminus\ol\om_1$), giving that $\om_1$ and $\om_2$ must be concentric balls and $u$ is radial in $\ol\om_2\setminus \om_1$. Now, by analyticity, $u$ is radial in the whole $\Om\setminus\ol D$, and, in particular, the level set $\pa\Om$ is a sphere. In other words, $\Om$ is a ball, and thus, we can use \cite[Theorem 5.1]{Sak bessatsu} to obtain that $D$ and $\Om$ are concentric balls, which is the desired result.
\end{proof}

\begin{remark}\label{rem:improvement of ThmI using B}
By applying \cite[Theorem I]{bessatsata} instead of \cite[Theorem 5.1]{Sak bessatsu} in the final step of the proof, we realize that the conclusion of Theorem \ref{(0,2) symmetry} holds even without assuming the $C^2$ regularity of $\partial D$.
\end{remark}

\begin{remark}\label{remark bello}
A different ``external" double overdetermination related to that of type $(0,2)$ has been considered in \cite{cava symm asymm}. In fact, \cite[Theorem I]{cava symm asymm} shows that the double overdetermination $u_{\nu}\equiv \text{const.}$, $u_{\nu\nu}\equiv \text{const.}$ on the outer boundary leads to symmetry in the two-phase setting. In the same paper, the author proves that the analogous $k$-fold overdetermination does not necessarily imply radial symmetry in the $k$-phase setting. 
    We stress that the construction in \cite[Theorem II]{cava symm asymm} shows that the presence of arbitrarily many overdetermined level sets in the outermost layer of a $k$-phase domain is not enough to obtain radial symmetry for $k\ge 3$. Thus, also in this sense, Theorem~\ref{(0,2) symmetry} can be considered sharp. 
    \end{remark}

\section{Symmetry results in particular cases for overdeterminations of type $(1,0)$ and $(1,1)^\star$ }\label{sec:particular cases}
We start by analyzing the case $(1,0)$, that is the case where the solution $u$ of \eqref{eq:Dirichlet problem} also satisfies
\begin{equation}\label{eq:internal overdetermination}
    u = a_1 , \quad u_\nu= c_1 \quad \text{ on } \pa \om_1, \quad \text{ where } \om_1 \subset\subset D .
\end{equation}

Let $\nu$ denote the exterior unit normal to both $D$ and $\Om$.
Throughout the present section we assume that $\Om$ and $D$ are of class $C^{2,\al}$ so that $u \in C^{2,\al}(\ol{\Om} \setminus D) \cap C^{2,\al}( \ol{D}) $ (see for instance \cite{XB2013, Zhuge}).

In what follows, we will make use of tools of tangential calculus on $\pa D$. To this aim, we recall the following definition. For $x\in \pa D$, $x_\tau$ will denote its tangential component, that is
$$
x_\tau := x- \langle x, \nu \rangle \nu \quad  \text{ on } \pa D .
$$
Given a function $f \in C^1 (\pa D)$ we define its tangential gradient by 
$$\gr_\tau f : = \na \Tilde{f} - \langle \na \Tilde{f}, \nu \rangle \nu  \quad  \text{ on } \pa D ,$$
where $\Tilde{f}$ is a $C^1$ extension of $f$ to a neighborhood of $\pa D$. It is easy to check and well-known that such a definition does not depend on the particular choice of the extension. Moreover, given a $C^1$ vector field $w =(w_1, \dots, w_N) : \pa D  \to \RR^N$, we denote by $D_\tau w$ the matrix whose $i$-th row is given by $\gr_\tau w_i$, for $i=1, \dots , N$.
Similarly, for $w \in C^1 (\pa D, \rn)$, the tangential divergence of $w$ is defined as $\dv_\tau w := \dv \Tilde w -\langle D\Tilde w\, \nu, \nu \rangle$, where $\Tilde w$ is a $C^1$ extension of $w$ to a neighborhood of $\pa D$.
The mean curvature of $\pa D$ will be denoted by $H:=\frac{1}{N-1}\dv_\tau \nu$. We remark that, following this definition, the mean curvature of the unit sphere is $H\equiv 1$. Moreover, in what follows, we will also make use of the Laplace--Beltrami operator $\De_\tau$, defined as $\dv_\tau\circ\gr_\tau$. Now we recall the following well-known identity for $f\in C^2(\ol D)$ (see for instance, \cite[Proposition 5.4.12]{HP2005} or \cite{Reilly}):
\begin{equation}\label{Reilly?}
    \De f= \De_\tau f +(N-1)H f_\nu + f_{\nu\nu} \quad \text{on }\pa D. 
\end{equation}
In the following lemma, we will present some general identities that will come in handy later on in our computations. 
\begin{lemma}\label{lem: jump}
Let $U$ be a neighborhood of $\pa D$. Then the following hold:  
\begin{enumerate}[(i)]
\item If $f$ belongs to either $C^1(\ol D\cap U)$ or $C^1(U\setminus D)$, then $\gr f= \gr_\tau f+ f_\nu \nu$ on $\pa D$.
\item If $f$ belongs to either $C^2(\ol D\cap U)$ or $C^2(U\setminus D)$, then $\gr^2 f\, \nu = f_{\nu\nu}\nu+ \gr_\tau f_\nu - D_\tau \nu \gr_\tau f$ on $\pa D$. 
\item Let $\ga\in\RR$. If $f$ satisfies $\De f= \ga$ in either $\ol D\cap U$ or $U\setminus D$ and is of class $C^2$ up to $\pa D$, then $f_{\nu\nu}=\ga - \De_\tau f - (N-1) H f_\nu$ on $\pa D$. 
\end{enumerate} 
Moreover, let $u$ denote a solution to $\dv(\sg\gr u)=N$ in $U$. Then the following hold:
\begin{enumerate}[(i)]\setcounter{enumi}{3}
\item If $u\in C^1(\ol D\cap U)\cap C^1(\ol U\setminus D)$, then $\jump{\sg\gr u}=\jump{\sg}\gr_\tau u$ on $\pa D$,
\item If $u\in C^2(\ol D\cap U)\cap C^2(\ol U\setminus D)$, then $\jump{\sg \gr^2 u\, \nu}=\jump{\sg u_{\nu\nu}}\nu-\jump{\sg}D_\tau \nu \gr_\tau u$ on $\pa D$,
\item If $u\in C^2(\ol D\cap U)\cap C^2(\ol U\setminus D)$, then $\jump{\sg u_{\nu\nu}}=-\jump{\sg}\De_\tau u$ on $\pa D$.

\end{enumerate}
\end{lemma}
\begin{proof}
Item $(i)$ immediately follows by the definition of tangential gradient $\gr_\tau$.

In order to show $(ii)$, we decompose $\gr^2 f \, \nu$ in its tangential and normal components:
\begin{equation*}
    \gr^2 f \, \nu = f_{\nu\nu}\nu + \left( \gr^2 f \, \nu \right)_\tau \quad \text{on } \pa D.
\end{equation*}
The claim will follow once we show that 
\begin{equation}\label{eq:olA expression}
\left( \gr^2 f \, \nu \right)_\tau= \na_\tau f_\nu - D_\tau \nu \na_\tau f \quad \text{on }\pa D.
\end{equation}
To this end, letting $\Tilde{\nu}$ denote a sufficiently smooth unitary extension of $\nu$ (for instance, given by the gradient of the signed distance function to $\pa D$) and setting $\Tilde{f}_\nu:= \langle \na f , \Tilde{\nu} \rangle $ and $\na \Tilde{f}_\nu = \na^2 f \Tilde{\nu} + D \Tilde{\nu} \na f$, we can compute that
\begin{equation*}
    \na_\tau f_\nu = \na \Tilde{f}_\nu - \langle \na \Tilde{f}_\nu  , \nu \rangle \nu 
    = \na^2 f \Tilde{\nu} + D_\tau \Tilde{\nu} \na_\tau f - f_{\Tilde{ \nu} \Tilde{ \nu}} \Tilde{ \nu} - \langle D \Tilde{\nu} \na f, \Tilde{\nu} \rangle \Tilde{\nu} \quad \text{on }\pa D,
\end{equation*}
from which, using that $D \Tilde{\nu} \Tilde{\nu} = \Orig$ (which easily follows by differentiating $ | \Tilde{\nu} |^2 \equiv 1 $), we obtain \eqref{eq:olA expression}.

Item $(iii)$ immediately follows from \eqref{Reilly?}.

Finally, items $(iv)$, $(v)$, and $(vi)$ follow from $(i)$, $(ii)$, and $(iii)$ respectively, applied to the restrictions of $\restr{\sg u}{\ol D\cap U}$ and $\restr{\sg u}{U\setminus D}$, by computing the jumps with the aid of the transmission condition $\jump{\sg u_\nu}= 0$ on $\pa D$.
\end{proof}

The following integral identity will be useful.
\begin{lemma}[Fundamental integral identity]\label{lem:fundamental identity having u explicit in D}
Let $u$ satisfy \eqref{eq:Dirichlet problem} and 
\begin{equation}\label{eq:forma di u in D}
	u(x) = \frac{|x|^2 - \la^2}{2 \sg_c } \quad \text{ in } D ,
\end{equation}
for some $\la>0$.
Then, for any $\xi \in \RR^N$, we have the following fundamental identity:
\begin{equation}\label{eq:fundamental identity with I II III}
\int_{\Om\setminus\ol{D}} (-u) \left\lbrace |\na^2 u|^2 - \frac{(\De u)^2}{N} \right\rbrace dx
= I + II + III ,
\end{equation}
where $|\na^2 u|$ denotes the Frobenius norm of the Hessian matrix of $u$ and 
\begin{equation}\label{eq:def I}
	I :=	\frac{1}{2} \int_{\pa\Om} u_\nu^2 \left[ u_\nu - \langle x - \xi , \nu \rangle \right] dS_x ,
\end{equation}	
\begin{equation}\label{eq: def II}
	\begin{split}
	II := \frac{1}{\sg_c} \left( \frac{1}{\sg_c} - 1 \right) 
	& \int_{\pa D} \Biggl\{
	\langle x , \nu \rangle \left( \langle x , \nu \rangle^2 - |x|^2 \right) 
	\Biggr.
	\\
	& 
	\Biggl. +
	\frac{\la^2 - |x|^2 }{2} \left[  \frac{\langle D_\tau \nu \,  x_\tau, x_\tau \rangle }{\sg_c} + (N-1) \left(1-H \langle x,\nu \rangle \right) \langle x,\nu \rangle \right] 
	 \Biggr\} dS_x ,
	\end{split}
\end{equation}	
\begin{equation}\label{eq:def III}
\begin{split}
	III:= 
    \int_{\pa D} \Biggl\{ 
    N u  \langle \xi , \nu \rangle
    &
    + \frac{ \langle \xi , \nu \rangle }{2}  \left[ \left( 1 -\frac{1}{\sg_c^2} \right) \langle x , \nu \rangle^2 +\frac{|x|^2}{\sg_c^2}  \right] 
	\Biggr.
    \\
    & -
	\Biggl.
    \left( 1 - \frac{1}{\sg_c} \right) \langle x , \nu \rangle^2  \langle \xi , \nu \rangle 
    -
    \frac{  \langle x , \nu \rangle \langle \xi , x \rangle }{\sg_c} 
	\Biggr\} dS_x .
 \end{split}
\end{equation}
\end{lemma}
\begin{proof}
Setting 
$$
P:= \frac{1}{2} |\na u|^2 - u
$$
and integrating by parts, we have that
\begin{equation*}
\int_{\Om \setminus \ol{D}} \left( P \De u - u \ \De P \right) dx = \int_{\pa\Om} \left( P u_\nu - u P_\nu \right) dS_x - \int_{\pa^+ D} \left( P u_\nu - u P_\nu \right) dS_x ,
\end{equation*}
where, as usual, on $\pa D$ and $\pa \Om$ we agree to denote by $\nu$ the unit normal exterior to $D$ and $\Om$. We also use the notation $\pa^+ D$ for the integral over $\pa^+ D$ whenever we need to emphasize that the values of (the derivatives of) $u$ over $\pa^+ D$ are those coming from $\Om \setminus \ol{D}$.
Using that $u=0$ on $\pa\Om$, the last formula easily leads to
\begin{equation}\label{eq:step 1 for Int Id}
    \int_{\Om \setminus \ol{D}} (- u) \De P  dx = - \int_{\Om \setminus \ol{D}}  P \De u \, dx + \int_{\pa\Om}  P u_\nu dS_x - \int_{\pa^+ D} \left( P u_\nu - u P_\nu \right) dS_x .
\end{equation}
Next, we are going to show that
\begin{multline}\label{eq:step 2 for Int Id}
    - \int_{\Om \setminus \ol{D}}  P \De u \, dx = 
    \int_{\pa^+ D} \left\lbrace \left[ \langle x - \xi , \na u  \rangle + (N-1) u  \right]  u_\nu  - 
 \left( \frac{|\na u|^2}{2} + N u \right) \langle x - \xi , \nu \rangle  \right\rbrace dS_x 
    \\
    - \frac{1}{2} \int_{\pa\Om} u_\nu^2 \langle x- \xi, \nu \rangle dS_x
    , 
\end{multline}
holds for any given $\xi \in \RR^N$.
The above equality follows from the divergence theorem, recalling that $\De u=N$ in $\Om \setminus \ol{D}$ and $u=0$ on $\pa\Om$, and using the following differential identities
\begin{equation*}
    \mathrm{div} (u \na u) = |\na u|^2 + ( \De u ) u ,
\end{equation*}
\begin{equation*}
    \mathrm{div} \left\lbrace \left( \langle x - \xi , \na u \rangle + N u \right) \na u - \left(  \frac{| \na u|^2}{2} + N u  \right) ( x - \xi )  \right\rbrace = \left( \frac{N}{2} + 1 \right) | \na u |^2 ,
\end{equation*}
which hold true in $\Om \setminus \ol{D}$.
Plugging \eqref{eq:step 2 for Int Id} into \eqref{eq:step 1 for Int Id} and using that
\begin{equation*}
P_\nu = \langle \na^2 u \na u , \nu \rangle - u_\nu \quad \text{ on } \pa^+ D ,
\end{equation*}
a direct computation gives that 
\begin{align*}
\int_{\Om \setminus \ol{D}} (- u) \De P  dx 
& = \frac{1}{2} \int_{\pa\Om} u_\nu^2 \left[ u_\nu - \langle x - \xi , \nu \rangle \right] dS_x
\\
& + 
\int_{\pa^+ D} \left[ u  \ \langle \na^2 u \na u , \nu \rangle - \frac{ | \na u |^2 }{2} u_\nu  \right] dS_x
\\
& +
\int_{\pa^+ D} \left\lbrace \left[ \langle x - \xi , \na u  \rangle + (N-1) u  \right]  u_\nu  - 
 \left( \frac{|\na u|^2}{2} + N u \right) \langle x - \xi , \nu \rangle  \right\rbrace dS_x .
\end{align*}
We now re-write the last formula using that
$$
\De P = | \na^2 u|^2 - \frac{(\De u)^2}{N} \quad \text{ in } \Om \setminus \ol{D} ,
$$   
which follows by direct computation, and that, by
\eqref{eq:forma di u in D} and the transmission conditions in \eqref{eq: transmission problem}, we have that
\begin{equation*}
	u(x) = \frac{|x|^2 - \la^2}{2 \sg_c}   \quad \text{ on\footnotemark  } \, \,  \pa D  ,
\end{equation*}
\begin{equation*}
	u_\nu (x) = \langle x,\nu \rangle \qquad \text{ on } \,  \pa^+ D ,
\end{equation*}	
\begin{equation*}
    \gr_\tau u(x)=\frac{x}{\sg_c}-\frac{\langle x,\nu\rangle}{\sg_c}\nu=\frac{x_\tau}{\sg_c} \qquad \text{ on } \,  \pa^+ D ,
\end{equation*}
\begin{equation*}
	\na u(x) = \frac{x}{\sg_c} - \frac{\langle x,\nu \rangle}{\sg_c} \nu + \langle x,\nu \rangle \nu   \quad \text{ on } \pa^+ D .
\end{equation*}
\footnotetext{Since this relation holds true on both $\pa^+ D$ and $\pa^- D$, we simply write $\pa D$.}
Hence, for any $\xi \in \RR^N$, we have that
\begin{equation*}
	\begin{split}
\int_{\Om\setminus\ol{D}} (-u) \left\lbrace |\na^2 u|^2 - \frac{(\De u)^2}{N} \right\rbrace dx
& =
\frac{1}{2} \int_{\pa\Om} u_\nu^2 \left[ u_\nu - \langle x - \xi , \nu \rangle \right] dS_x
\\
& + 
\int_{\pa^+ D} \Biggl\{  (-u) \left[ N \langle x - \xi , \nu \rangle - (N-1) \langle x , \nu \rangle - \langle \na^2 u \na u , \nu \rangle \right] 
\Biggr.
\\
& -
\frac{1}{2} \left( \langle x - \xi , \nu \rangle + \langle x , \nu \rangle  \right) \left[ \left( 1 -\frac{1}{\sg_c^2} \right) \langle x , \nu \rangle^2 +\frac{|x|^2}{\sg_c^2}  \right]
\\
& 
\Biggl. +
\langle x , \nu \rangle \left[ \left( 1 - \frac{1}{\sg_c} \right) \langle x - \xi , \nu \rangle  \langle x , \nu \rangle +\frac{\langle x - \xi , x \rangle }{\sg_c}  \right] \Biggr\} dS_x ,
	\end{split}
\end{equation*}
which can be easily rearranged, by simple computations that use that $\langle x - \xi , \nu \rangle = \langle x , \nu \rangle - \langle \xi , \nu \rangle$ to gather the terms (on $\pa D$) depending on $\xi$ in $III$ below, as follows:
\begin{equation*}
		\int_{\Om\setminus\ol{D}} (-u) \left\lbrace |\na^2 u|^2 - \frac{(\De u)^2}{N} \right\rbrace dx
		 =
		I + II + III ,
\end{equation*}
where we have set $I$ and $III$ as in \eqref{eq:def I} and \eqref{eq:def III}, and

\begin{equation*}
	II := \int_{\pa^+ D} \left\lbrace  (-u) \left[  \langle x , \nu \rangle - \langle \na^2 u \na u , \nu \rangle \right] 
	+
	\frac{1}{\sg_c} \left( \frac{1}{\sg_c} - 1 \right) \langle x , \nu \rangle \left( \langle x , \nu \rangle^2 - |x|^2 \right) \right\rbrace dS_x .
\end{equation*}	

We finally show that $II$ can be conveniently re-written as in \eqref{eq: def II}, again by exploiting the transmission conditions in \eqref{eq: transmission problem} and the fact that
\begin{equation*}
	u(x) = \frac{|x|^2 - \la^2}{2 \sg_c} \quad \text{ in } D .
\end{equation*}
More precisely, we are going to prove that
\begin{equation}\label{eq:proof second derivatives on pa D}
	\langle \na^2 u \na u, \nu \rangle = \langle x,\nu \rangle + \left( 1 - \frac{1}{\sg_c} \right) \left\lbrace \frac{ \langle D_\tau \nu \,  x_\tau, x_\tau \rangle }{\sg_c} + (N-1) \left(1-H \langle x,\nu \rangle \right) \langle x,\nu \rangle \right\rbrace ,
\end{equation}
where $ x_\tau = x - \langle x,\nu \rangle \nu $.
Once \eqref{eq:proof second derivatives on pa D} is proved, it is immediate to check that $II$ can be conveniently re-written as in \eqref{eq: def II}, which completes the proof.
Notice that, since we got rid of the presence of $\langle \na^2 u \na u , \nu \rangle$, the integral over $\pa^+D$ can be simply denoted by $\pa D$. 

 In order to prove \eqref{eq:proof second derivatives on pa D}, we recall that $\si := \sg_c \cX_D + \cX_{\Om \setminus D} $ and that $\jump{u} = \jump{\na_\tau u} = \jump{\De_\tau u}=0=\jump{\si u_\nu}$, where $\jump{\cdottone}$ denotes the ``jump'' across $\pa D$.
We start by computing that, on $\pa^+ D$:
\begin{equation}\label{eq:initial step to compute second der on partial D}
    \langle \na^2 u \na u, \nu \rangle 
    = \langle \na^2 u \na_\tau u, \nu \rangle + u_{\nu \nu} u_\nu 
    = \frac{ \langle \na^2 u \  x, \nu \rangle}{\sg_c} + u_{\nu \nu} \ \langle x, \nu \rangle \left( 1 - \frac{1}{\sg_c} \right).
\end{equation}
Recalling \eqref{eq:forma di u in D} and $(vi)$ of Lemma~\ref{lem: jump}, we obtain that
\begin{equation}\label{eq:extra for second der on partial D}
\jump{\si u_{\nu \nu}} 
= - \jump{\si} \De_\tau \left( \frac{|x|^2 - \la^2}{2 \si_c} \right)
= - \frac{\jump{\sg}}{\sg_c} \De_\tau \left(\frac{|x|^2}{2}\right) 
= - \frac{\jump{\sg}}{\sg_c} (N-1) \left( 1 - H \langle x, \nu \rangle  \right) ,
\end{equation}
from which we easily obtain that
\begin{equation}\label{eq:u_nunu on partialD}
  u_{\nu \nu} = 1+ (N-1) \left(  1 - \frac{1}{\sg_c}  \right)  \left( 1 - H \langle x, \nu \rangle  \right) \quad \text{ on } \pa^+ D .
\end{equation}
Next, we compute that
\begin{equation}\label{eq:second step for second der on partial D}
\begin{split}
    \jump{  \langle \si   \na^2 u \  x, \nu \rangle } 
    & = \jump{   \langle \si \na^2 u \  x_\tau , \nu \rangle } + \jump{   \langle \si \na^2 u \  \langle x, \nu \rangle \nu , \nu \rangle }
    \\
    & =  \langle \jump{   \si \na^2 u \ \nu  } , x_\tau \rangle +  \jump{    \si \  u_{\nu \nu}  } \  \langle x, \nu \rangle
    \\
    & = \langle \jump{   \si \na_\tau u_\nu  } , x_\tau \rangle - \langle \jump{ \si D_\tau \nu \  \na_\tau u } , x_\tau \rangle +  \jump{    \si \  u_{\nu \nu}  } \  \langle x, \nu \rangle
    \\
    & =  \langle \na_\tau\jump{   \si  u_\nu  } , x_\tau \rangle -  \jump{ \si } \langle D_\tau \nu \ \na_\tau u  , x_\tau \rangle - \frac{ \jump{ \si } }{ \si_c }   (N-1) \left( 1 - H \langle x, \nu \rangle \right) \  \langle x, \nu \rangle
    \\
    & = -\frac{ \jump{ \si } }{\si_c} \langle D_\tau \nu \ x_\tau  , x_\tau \rangle - \frac{ \jump{ \si } }{ \si_c }   (N-1) \left( 1 - H \langle x, \nu \rangle \right) \  \langle x, \nu \rangle .
\end{split}
\end{equation}
Here, the third equality follows by $(v)$ of Lemma~\ref{lem: jump}, the fourth equality follows by \eqref{eq:extra for second der on partial D}, whereas in the fifth equality we used that $\jump{\si u_\nu} = 0$ and $\na_\tau u = x_\tau / \si_c$.

Putting together \eqref{eq:initial step to compute second der on partial D}, \eqref{eq:u_nunu on partialD}, and \eqref{eq:second step for second der on partial D}, we obtain \eqref{eq:proof second derivatives on pa D} and complete the proof.
\end{proof}

The following theorem provides general necessary and sufficient conditions for the rigidity of Problem \eqref{eq:Dirichlet problem} under the condition \eqref{eq:internal overdetermination}, that is,
\begin{equation*}
    u = a_1 , \quad u_\nu= c_1 \quad \text{ on } \pa \om_1 , \quad \text{ where } \om_1 \subset\subset D .
\end{equation*}
\begin{theorem}\label{thm:symmetry for internal only}
    Let $u$ satisfy \eqref{eq:Dirichlet problem} and \eqref{eq:internal overdetermination}. Assume\footnote{\label{footnote translation}Such an assumption is always satisfied, up to a translation.} that the origin $\Orig$ coincides with the center of mass of $\om_1$. Then, the following items are equivalent:
\begin{enumerate}[(i)]
    \item $D$ and $\Om$ are concentric balls, and, up to a dilation and a translation, $u$ is of the form
    \begin{equation*}
        u(x)=
        \begin{cases}
        \frac{|x|^2 - \la^2}{2 \sg_c} \quad & \text{ in } D = B_1 (\Orig) ,
        \\
        \frac{|x|^2 - R^2}{2} \quad & \text{ in } \Om = B_R (\Orig) ,
        \end{cases}
    \end{equation*}
    where $R>1$ and $\la^2 = \sg_c R^2 +1 -\sg_c$.
    \item $D$, $\Om$, and $u_\nu$ on $\pa\Om$ are such that the following inequality is satisfied:
    \begin{equation*}
    \begin{split}
    \frac{1}{2} \int_{\pa\Om} u_\nu^2 \left[ u_\nu - \langle x , \nu \rangle \right] & dS_x 
    \\
    + 
    \left( \frac{1}{\sg_c} - 1 \right) 
	& \int_{\pa D} \Biggl\{
	\frac{ \langle x , \nu \rangle }{\sg_c} \left( \langle x , \nu \rangle^2 - |x|^2 \right) 
	\Biggr.
	\\
	& 
	\Biggl. 
    - u
    \left[  \frac{\langle D_\tau \nu \,  x_\tau, x_\tau \rangle }{\sg_c} + (N-1) \left(1-H \langle x,\nu \rangle \right) \langle x,\nu \rangle \right] 
	 \Biggr\} dS_x \le 0 .
    \end{split}
    \end{equation*}
    \end{enumerate}
\end{theorem}
\begin{proof}[Proof of Theorem~\ref{thm:symmetry for internal only}]
By the classical Serrin symmetry result in $\om_1$ and analytic continuation, we find that $u$ is of the form
\begin{equation*}
	u(x) = \frac{|x|^2 - \la^2}{2 \sg_c} \quad \text{ in } D ,
\end{equation*}
for some $\la>0$.
Hence, we are in a position to apply Lemma~\ref{lem:fundamental identity having u explicit in D}. We thus use \eqref{eq:fundamental identity with I II III} with $\xi=\Orig$ to find that
$$
\int_{\Om\setminus\ol{D}} (-u) \left\lbrace |\na^2 u|^2 - \frac{(\De u)^2}{N} \right\rbrace dx
=
\frac{1}{2} \int_{\pa\Om} u_\nu^2 \left[ u_\nu - \langle x , \nu \rangle \right] dS_x + II ,$$
where $II$ is as in \eqref{eq: def II}.
We now show that $(i)$ and $(ii)$ are equivalent. 

If $(i)$ holds true, then it is immediate to check that
$\left\{ | \na^2 u |^2 - \frac{ ( \De u )^2 }{N} \right\} \equiv 0$ in $\Om \setminus \ol{D}$ and hence $(ii)$ follows; in fact, $(ii)$ holds true with the equality sign.

On the other hand, if (ii) holds true, then we have that
$$
\int_{\Om\setminus\ol{D}} (-u) \left\lbrace |\na^2 u|^2 - \frac{(\De u)^2}{N} \right\rbrace dx = 0 
$$
and hence
$$
|\na^2 u|^2 - \frac{(\De u)^2}{N} \equiv 0 \text{ in } \Om \setminus \ol{D} ,
$$
being as $-u >0$ by the maximum principle (Lemma~\ref{strong maximum pple}), and
\begin{equation*}
|\na^2 u|^2 - \frac{(\De u)^2}{N} =  |\na^2 u|^2 - \frac{\langle \na^2 u, \identmatrix \rangle^2_{\RR^{N^2}}}{|\identmatrix|^2}   
\ge 0 
\end{equation*}
 by the Cauchy--Schwarz inequality\footnote{Obtained by regarding the matrices $\gr^2 u$ and the identity matrix $\identmatrix$ in $\RR^{N\times N}$ as vectors in $\RR^{N^2}$.}  in $\RR^{N^2}$. In particular, the Cauchy--Schwarz inequality holds with the equality sign, and hence (see, for instance, \cite[Lemma~1.9]{PogTesi} or \cite{MP2}) $u$ is a quadratic polynomial of the form
$$ 
\frac{|x-\eta|^2 - R^2}{2} \quad \text{ in } \Om \setminus \ol{D} ,
$$
for some $\eta \in \RR^N$ and $R>0$. The transmission condition of $u_\nu$ on $\pa D$ readily gives that
$$
\langle x,\nu \rangle = \langle x - \eta ,\nu \rangle \quad \text{ for any } x \in \pa D,
$$
and hence $\eta = \Orig $. 
Hence, $\Om$ is a ball of radius $R$ centered at the origin, i.e., $\Om = B_R(O)$. Moreover, the transmission condition of $u$ on $\pa D$ gives that 
$$
\frac{ |x|^2 - \la^2 }{2 \sg_c}  = \frac{|x|^2 - R^2}{2} \quad \text{ on } \pa D,
$$
that is, 
\begin{equation}\label{eq:STAR_D is a ball} 
|x|^2 = \frac{\sg_c}{\sg_c-1} \left( R^2 - \frac{\la^2}{\sg_c} \right) \quad \text{ for any } x \in \pa D .
\end{equation}
Hence, also $D$ is a ball centered at the origin.
Thus, $(i)$ follows, and the equivalence of $(i)$ and $(ii)$ is proved.
\end{proof}

We now focus on overdetermination of type $(1,1)^\star$, that is when, in addition to \eqref{eq:Dirichlet problem} and \eqref{eq:internal overdetermination}, $u$ also satisfies 
\begin{equation}\label{eq:external overdetermination}
u_\nu=c_2\quad \text{on } \pa\Om.
\end{equation}

In light of the discussion of section \ref{sec:preliminary}, the following result implies Theorem~\ref{(1,1) D=ball symmetry}. 
\begin{theorem}\label{thm:InandOut plus D ball sufficient condition}
Let $u$ satisfy \eqref{eq:Dirichlet problem} together with \eqref{eq:internal overdetermination} and \eqref{eq:external overdetermination}. If $D$ is a ball, then $\Om$ must be a concentric ball, and, up to a dilation and a translation, $u$ is of the form
    \begin{equation*}
        u(x)=
        \begin{cases}
        \frac{|x|^2 - \la^2}{2 \sg_c} \quad & \text{ in } D = B_1 (\Orig) ,
        \\
        \frac{|x|^2 - R^2}{2 } \quad & \text{ in } \Om = B_R (\Orig) ,
        \end{cases}
    \end{equation*}
    where $R>1$ and $\la^2 = \sg_c R^2 +1 -\sg_c$.
\end{theorem}
\begin{proof}
Without loss of generality, up to a dilation, we can assume $D$ to be a ball of radius $1$, and, up to a translation, we can fix the origin $\Orig$ in the center of mass of $\pa\om_1$. In this way, we have that $D=B_1(z)$ for some $z \in \RR^N$, and $u$ is of the form \eqref{eq:forma di u in D}, that is,
\begin{equation*}
	u(x) = \frac{|x|^2 - \la^2}{2 \sg_c} \quad \text{ in } D ,
\end{equation*}
for some $\la>0$.

Thus, Lemma~\ref{lem:fundamental identity having u explicit in D} applies and \eqref{eq:fundamental identity with I II III} holds true.
Notice that,
by \eqref{eq:external overdetermination}, the divergence theorem, and the fact that $\dv(\sg \gr u) =N$ in $\Om$, we have that $$
I := \frac{1}{2} \int_{\pa\Om} u_\nu^2 \left[ u_\nu - \langle x - \xi , \nu \rangle \right] dS_x
= \frac{c_2^2}{2} \int_{\Om} \left[ \dv(\sg \gr u) - \dv( x - \xi ) \right] dx
= 0,
$$
regardless of the choice of $\xi \in \RR^N$. Hence, \eqref{eq:fundamental identity with I II III} reduces to 
\begin{equation}\label{eq:new integral identity senza I}
\int_{\Om\setminus\ol{D}} (-u) \left\lbrace |\na^2 u|^2 - \frac{(\De u)^2}{N} \right\rbrace dx
= II + III , 
\end{equation}
where $II$ and $III$ are those defined in \eqref{eq: def II} and \eqref{eq:def III}. 

Since the left-hand side and $II$ do not depend on $\xi$, we must have that
$$\na_\xi III = \Orig ,$$
that is, by direct computation,
\begin{equation*}
    \Orig = \na_\xi III  = 
    \int_{\pa D} \Biggl\{ 
    N u \, \nu 
    +  \left[ \left( \frac{1}{\sg_c} - \frac{1}{2} -\frac{1}{2 \sg_c^2}
    \right) \langle x , \nu \rangle^2 +\frac{|x|^2}{2 \sg_c^2}  \right] \, \nu
    -
    \frac{  \langle x , \nu \rangle  }{\sg_c} \, x 
	\Biggr\} dS_x .
\end{equation*}
By the divergence theorem $\int_{\pa D} u \, \nu \,dS_x =  \int_{D} \na u \, dx$ and using that $\na u = x / \sg_c$ n $D$ (by \eqref{eq:forma di u in D}), we thus get that
$$
\int_{\pa D} \left\{
    N u  +  \frac{|x|^2}{2 \sg_c^2} \right\}\nu\, dS_x 
= \left( \frac{N}{\sg_c} + \frac{1}{\sg_c^2} \right) \int_D x \,dx
$$
and hence, the formula above can be re-written as follows:
\begin{equation}\label{eq:intermedia}
    \Orig = \na_\xi III 
    = 
    \left( \frac{N}{\sg_c} + \frac{1}{\sg_c^2} \right) \int_D x \,dx 
    \\ +
    \int_{\pa D} \Biggl\{ \left( \frac{1}{\sg_c} - \frac{1}{2} -\frac{1}{2 \sg_c^2}
    \right) \langle x , \nu \rangle^2   \, \nu
    -
    \frac{  \langle x , \nu \rangle  }{\sg_c} \, x 
	\Biggr\} dS_x .
\end{equation}
We stress that we have not used yet that $D$ is a ball. We now use that
$D=B_1(z)$, and hence that $\nu = x-z$, and we compute that
\begin{equation*}
    \int_D x \,dx = z |B_1| ,
\end{equation*}
\begin{equation*}
\begin{split}
    \int_{\pa D} \langle x , \nu \rangle^2 \, \nu \, dS_x 
    & = \int_{\pa D} \left( 1 + \langle z , x-z \rangle^2 + 2 \langle z , x-z \rangle   \right) \, \nu \, dS_x 
    \\
    & = 2 z \, \int_D  ( \langle z , x-z \rangle + 1 )  dx 
    \\
    & = 2 z |B_1| ,
    \end{split}
\end{equation*}
\begin{equation*}
\begin{split}
   \int_{\pa D}   \langle x , \nu \rangle   \, x \, dS_x 
   & = \int_{\pa D}   \langle x , \nu \rangle   \, (x-z) \, dS_x + z \int_{\pa D}   \langle x , \nu \rangle  dS_x 
   \\
   & = \int_{\pa D}   \langle x , x-z \rangle   \, \nu \, dS_x + N z |B_1| 
   \\
   & = \int_{D} (2x-z) dx + N z |B_1| 
   \\
   & = (N+1) z |B_1| .
   \end{split}
\end{equation*}
In the above formulas, we used the divergence theorem and that $\int_D  \langle z , x-z \rangle dx=0$ by symmetry. Plugging these formulas in \eqref{eq:intermedia} easily leads to
$$
\Orig = \na_\xi III = \left( \frac{1}{\sg_c} - 1 \right) z |B_1| ,
$$
from which we deduce that we must have 
$ z=\Orig $; hence, we have that $u$ is constant on $\pa B_1 ( \Orig ) = \pa D$, and the symmetry result immediately follows by Proposition~\ref{prop:special case pa D level set} and Remark \ref{rem:special case pa D level set}.
\end{proof}
The following alternative proof is longer but shows that when $D$ is a ball, $II$ and $III$ can be explicitly computed.
\begin{proof}[Alternative proof of Theorem~\ref{thm:InandOut plus D ball sufficient condition}]
As in the previous proof, we arrive at \eqref{eq:new integral identity senza I}. We now explicitly compute $II$ and $III$.
Using that $D=B_1(z)$, we can directly check that, on $\pa D$:
\begin{equation*}
\nu = x-z ,
\end{equation*}
\begin{equation*}
H \equiv 1 ,
\end{equation*}
\begin{equation*}
	\langle x , \nu \rangle = 1 + \langle z , \nu \rangle ,
\end{equation*}
\begin{equation*}
\langle x , \nu \rangle^2 - |x|^2 = \langle z , \nu \rangle^2 - |z|^2 .
\end{equation*}
Noting, by direct computation, that
$\identmatrix - D_\tau \nu$ (where $\identmatrix$ denotes the identity matrix in $\RR^{N\times N}$) is the matrix whose $i$-th line is given by the vector $(x_i - z_i) (x-z)$, we easily compute that
$D_\tau \nu \,  x_\tau = x_\tau$,
and hence,
\begin{equation*}
\langle D_\tau \nu \,  x_\tau, x_\tau \rangle = \langle x_\tau, x_\tau \rangle = |x|^2 - \langle x,\nu \rangle^2 = |z|^2 - \langle z,\nu \rangle^2 .
\end{equation*}
Using the above information, tedious but easy computations give that (when $D=B_1(z)$) $II$ reduces to:
\begin{equation*}
	\begin{split}
II 
& = \frac{1}{\sg_c}
\left( \frac{1}{\sg_c} - 1 \right) 
\Biggl\{
- |z|^2 |\pa B_1| 
- \left( 1 + \frac{1}{\sg_c} \right) |z|^2 \int_{\pa D} \langle z , \nu \rangle dS_x 
\Biggr.
\\
& + N \int_{\pa D} \langle z , \nu \rangle^2 dS_x
+ \left( 2 - N - \frac{1}{\sg_c} \right) \int_{\pa D} \langle z , \nu \rangle^3 dS_x
\\
& + 
\frac{\la^2 - |z|^2 - 1 }{2} 
\left[ \frac{|z|^2}{\sg_c} |\pa B_1| 
- (N-1) \int_{\pa D} \langle z , \nu \rangle dS_x
- \left( N - 1 + \frac{1}{\sg_c} \right) \int_{\pa D} \langle z , \nu \rangle^2 dS_x .
 \right] 
 \Biggr\}
 \end{split}
\end{equation*}
Hence, using that, by symmetry, 
$$
\int_{\pa D} \langle z , \nu \rangle dS_x = 0 = \int_{\pa D} \langle z , \nu \rangle^3 dS_x ,
$$
and, by the divergence theorem,
$$
\int_{\pa D} \langle z , \nu \rangle^2 dS_x = \int_{\pa D} \langle \langle z, x-z \rangle z , \nu \rangle dS_x = \int_D \dv(\langle z, x-z \rangle z) dx = |z|^2 |B_1| ,
$$
by recalling that 
\begin{equation}\label{eq:Ball_volume_surface}
|\pa B_1| = N|B_1|
\end{equation}
we finally get that
\begin{equation}\label{eq:II}
	II= \frac{1}{\sg_c} \left( \frac{1}{\sg_c} - 1 \right)^2 (N-1) \, \frac{\la^2 - |z|^2 - 1 }{2} \, |B_1| |z|^2 .
\end{equation}

We are left to compute $III$ (using that $D=B_1(z)$). Writing $u$ as
$$
u= \frac{ \langle z, \nu \rangle }{\sg_c} - \frac{\la^2 - |z|^2 -1}{ 2 \sg_c}  \quad \text{on }\pa D,
$$
using (tedious but easy) manipulations similar to those used to compute $II$, and noting that, by symmetry,
$$\int_{\pa D} \langle \xi , \nu \rangle dS_x = 0 = \int_{\pa D} \langle \xi , \nu \rangle^3 dS_x ,$$
and, by the divergence theorem,
$$\int_{\pa D} \langle z , \nu \rangle \langle \xi , \nu \rangle dS_x = \int_{\pa D} \langle \langle z, x-z \rangle \xi , \nu \rangle dS_x = \int_D \dv(\langle z, x-z \rangle \xi) dx = \langle z, \xi \rangle |B_1| ,$$
we find that
\begin{equation}\label{eq:explicit III in case D ball}
III = \left( \frac{1}{\sg_c} - 1 \right) \langle z, \xi \rangle |B_1| .
\end{equation}

Choosing $\xi = \mu z $,
with
$$
\mu:= \frac{1}{\sg_c} \left( 1 -\frac{1}{\sg_c} \right) (N-1) \frac{\la^2 - |z|^2 - 1 }{2} ,
$$
gives that $II+III =0$; hence, \eqref{eq:new integral identity senza I} gives that
$$
\int_{\Om\setminus\ol{D}} (-u) \left\lbrace |\na^2 u|^2 - \frac{(\De u)^2}{N} \right\rbrace dx = 0 ,
$$
and we can reason as in the proof of Theorem~\ref{thm:symmetry for internal only} to get that \eqref{eq:STAR_D is a ball} holds true.
Since $D=B_1(z)$, 
we must have $z=\Orig$, $D=B_1( \Orig )$, and 
$\la^2 = \sg_c R^2 + 1 - \sg_c$.
\end{proof}
\begin{remark}
{\rm As a sanity check, it is immediate to compute from \eqref{eq:explicit III in case D ball} that $\na_\xi III = \left( \frac{1}{\sg_c} - 1 \right)  z |B_1|$, which agrees with the value obtained in the first of the two proofs.}
\end{remark}
Theorem~\ref{(1,1) D=ball symmetry} now easily follows from Theorem~\ref{thm:InandOut plus D ball sufficient condition} in light of the discussion of section \ref{sec:preliminary}.
\begin{proof}[Proof of Theorem~\ref{(1,1) D=ball symmetry}]
 As also remarked in section \ref{sec:preliminary}, the interior regularity of $u$ guarantees that $\pa \om_2$ is an analytic surface contained in $\Om\setminus\ol D$. Thus, Theorem~\ref{thm:InandOut plus D ball sufficient condition} applied to $(D,\om_2)$ yields that $(D,\om_2)$ are concentric balls and $u$ is radial up to $\ol\om_2$. By the analyticity of $u$ in $\Om\setminus\ol D$, and the fact that $\pa \Om$ is made of regular points for the Dirichlet Laplacian, we obtain that $u$ is radial in the whole $\ol\Om$ and $\pa\Om$, being a level set, must be a sphere concentric with $D$. This concludes the proof.      
\end{proof}

\section{Counterexamples for overdetermination of type $(1,1)^\star$}\label{sec:counter(1,1)}
We notice that the solvability of \eqref{eq:Dirichlet problem} under overdetermination of type $(1,1)^\star$ is equivalent to that of the following overdetermined problem in an annular domain $\Om\setminus\ol D$:
\begin{equation}\label{new odp in annular config}
   \begin{cases}
       \De u= N\quad \tin \Om\setminus\ol D,\\
       u=\frac{|x|^2-T}{2\sg_c}\quad \ton\pa D,\\
       u=0 \quad \ton \pa\Om,\\
   \end{cases} 
\end{equation}
for some real constant $T$, with overdetermined conditions 
\begin{equation}\label{new odc}
u_\nu= \langle x, \nu\rangle \quad \ton \pa D, \quad u_\nu= \text{const.} \quad \ton \pa\Om,
\end{equation}
where $\nu$ denotes the outer unit normal to $D$ at $\partial D$ and to $\Om$ at $\partial \Om$, respectively.

In what follows, we will find a nontrivial pair $(D,\Om)$ such that the solution to \eqref{new odp in annular config} also satisfies \eqref{new odc}.
\subsection{Preliminary result: a general bifurcation lemma}
Let $(D_0, \Om_0)$ be the pair of open balls of radii $R$ ($0<R<1$) and $1$ respectively centered at the origin. Moreover, let $Y_{k,i}$ denote the so-called \emph{spherical harmonics}, defined as the solutions to the following eigenvalue problem for the Laplace--Beltrami operator on the unit sphere
\begin{equation*}
-\De_\tau Y_{k,i}= \la_{k} Y_{k,i} \quad \ton \pa\Om_0,
\end{equation*}
where the eigenvalues are given by $\la_k=k(k+N-2)$, for $k\in \NN\cup\{0\}$, and the eigenfunctions are normalized such that $\norm{Y_{k,i}}_{L^2(\pa\Om_0)}=1$. Furthermore, the eigenspace $\cY_k$ corresponding to the $k$\textsuperscript{th} eigenvalue $\la_k$ has finite dimension $d_k$ and is spanned by $\left\{ Y_{k,1}, \dots, Y_{k,d_k}\right\}$.

The following Lemma gives sufficient conditions that ensure the existence of a nontrivial branch of solutions to an overdetermined problem near the trivial solution given by the spherical annulus.
\begin{lemma}[Bifurcation from an annular configuration]\label{lem bifurcation}
Let $(D_0,\Om_0)$ be defined as above. Let $P\subset \RR$ be an open set, and let $X_i$, $Y_i$ ($i=1,2$) be $O(N)$-invariant Banach spaces of real-valued functions defined on $\pa D_0$ for $i=1$ and on $\pa\Om_0$ for $i=2$. 
Here, we say that a space $W$ of functions defined on a sphere centered at the origin is $O(N)$-invariant if and only if $w\circ\ga\in W$ holds for all $w\in W$ and for all elements $\ga$ of the orthogonal group $O(N)$. 
Assume that the inclusions $X_i\subset Y_i$ hold with compact embeddings $\iota_i: X_i\hookrightarrow Y_i$ ($i=1,2$), let $X:=X_1\times X_2$, $Y:=Y_1\times Y_2$, and let $\iota^\pm$ denote the compact operators $(\eta,\xi)\mapsto (\pm\iota_1\eta,\iota_2\xi)$ between the product spaces $X\hookrightarrow Y$. 
Also, assume that for all $k\in \NN\cup\{0\}$ and  $i=1,\dots, d_k$, one has $Y_{k,i}(\ \cdottone/R)\in X_1$ and $Y_{k,i}(\cdottone)\in X_2$. 
Let 
\begin{equation*}
    F: X\times P\to Y
\end{equation*}
be a $C^\ell$ mapping ($3\le \ell\le \infty$). Assume that $F$ is $O(N)$-equivariant, that is, 
\begin{equation*}
   F(\eta\circ \ga, \xi\circ \ga, \rho) = F(\eta,\xi,\rho)\circ \ga
\end{equation*}
for all $(\eta,\xi)\in X$, $\rho\in P$ and $\ga\in O(N)$. Also, for the sake of notational simplicity, for all $\rho\in P$, let $L(\rho):X\to Y$ denote the partial Fr\'echet derivative 
\begin{equation*}
 X\ni (\eta,\xi)\mapsto L(\rho)[\eta,\xi]:= \pa_{X}F(0,0,\rho)[\eta,\xi].
\end{equation*}
Moreover, assume that the following hold for some pair $(\rho^\star,k^\star)\in P\times (\NN\cup\{0\})$:
\begin{enumerate}[(a)]
    \item $F(0,0,\rho)=(0,0)$ for all $\rho\in P$.
    \item There exists a real constant $\mu\in \RR$ such that at least one of the two maps 
    \begin{equation*}
        L(\rho^\star) + \mu \iota^\pm: X\to Y 
    \end{equation*}
    is a bounded bijection. 
    \item For $k\in \NN\cup\{0\}$, $i=1,\dots, d_k$, consider the 2-dimensional vector space
    \begin{equation*}
        X_{k,i}:=\setbld{\left( \be Y_{k,i}\left({\; \cdottone}/{R}\right), \ga Y_{k,i}(\cdottone)  \right)}{\be,\ga\in\RR}. 
    \end{equation*}
    Also, let $\psi_{k,i}:X_{k,i}\to\RR^2$ denote the linear isomorphism 
    \begin{equation*}
        \left( \be Y_{k,i}\left(\,{\cdottone} /{R}\right), \ga Y_{k,i}(\cdottone)  \right)\mapsto \begin{pmatrix}
            \be\\ \ga
        \end{pmatrix}.
    \end{equation*}  
    Under this notation, suppose that, for all $\rho\in P$, the restriction $\restr{L(\rho)}{X_{k,i}}$ maps $X_{k,i}$ into itself and that there exists a matrix-valued function $\cM:P\times \NN\cup \{0\}\to \RR^{2\times 2}$ such that 
    \begin{equation*}
\psi_{k,i}\left(L(\rho)[\eta,\xi]\right)=\cM(\rho, k)\psi_{k,i}(\eta,\xi), \quad \text{for all }(\eta,\xi)\in X_{k,i}. 
\end{equation*}
    \item For $k\in \NN\cup\{0\}$, $\det\cM(\rho^\star,k)=0$ if and only if $k=k^\star$. Moreover, $\cM(\rho^\star,k^\star)\ne 0$.
    \item $\det\pa_\rho\cM(\rho^\star,k^\star)\ne 0$.
\end{enumerate}
Then there exists a function $Y^{\star}\in\cY_{k^\star}$, two real constants $(\be,\ga)\ne (0,0)$ and a nontrivial branch of class $C^{\ell-2}$ 
\begin{equation*}
   (-\varepsilon,\varepsilon)\ni t\mapsto \left( \eta(t), \xi(t), \rho(t) \right)\in X_1\times X_2\times  P 
\end{equation*}
such that $\rho(0)=\rho^\star$, $\eta(0)=0$, $\xi(0)=0$,  $\eta'(0)=\be Y^{\star}(\,{\cdottone}/{R})$, $\xi'(0)=\ga Y^{\star}(\cdottone)$ and $F(\eta(t), \xi(t), \rho(t))=0$ for all $t\in (-\ve,\ve)$.
\end{lemma}

The proof of Lemma~\ref{lem bifurcation} relies on the following version of the Crandall--Rabinowitz bifurcation theorem (that is equivalent to the one stated in \cite{CR}). 

\begin{thm}[Crandall--Rabinowitz theorem]\label{Crandall--Rabinowitz theorem}
Let $X$, $Y$ be real Banach spaces and let $U\subset X$ and $P\subset \RR$ be open sets, such that $0\in U$.
Let $\Psi\in C^\ell(U\times P;Y)$ ($3\le\ell\le\infty$) and assume that there exist $\rho^\star\in P$ and $x^\star\in X$ such that 
\begin{enumerate}[(i)]
    \item $\Psi(0,\rho) = 0$ for all $\rho\in P$;
    \item ${\rm Ker}\ \pa_x \Psi(0,\rho^\star)$ is a 1-dimensional subspace of $X$, spanned by $x^\star$; 
    \item ${\rm Im}\ \pa_x \Psi(0,\rho^\star)$ is a closed co-dimension 1 subspace of $Y$;
    \item  $\pa_\rho\pa_x \Psi(0,\rho^\star)[x^\star]\notin {\rm Im}\ \pa_x \Psi(0,\rho^\star)$.
    \end{enumerate}
Then $(0, \rho^\star)$ is a bifurcation point of the equation $\Psi(x,\rho)=0$ in the following sense. 
In a neighborhood of $(0, \rho^\star)\in X\times P$, the set of solutions of $\Psi(x,\rho) = 0$ consists of two $C^{\ell-2}$-smooth curves $\Gamma_1$ and $\Gamma_2$ which intersect only at the point $(0,\rho^\star)$. $\Gamma_1$ is the curve $\{(0,\rho)\,:\,\rho\in P\}$ and $\Gamma_2$ can be parametrized as follows, for small $\ve>0$: 
\begin{equation*}
    (-\ve,\ve)\ni t \mapsto\left(x(t),\rho(t)\right)\in U\times P,\text{ such that }\left(x(0),\rho(0)\right)=(0,\rho^\star), \quad x'(0)=x^\star. 
\end{equation*}
\end{thm}

\begin{proof}[Proof of Lemma~\ref{lem bifurcation}]
We would like to apply Theorem~\ref{Crandall--Rabinowitz theorem} to the function $F$ but we cannot do this directly because $\dim \ker \pa_{X} F(0,0,0) \ne 1$ in general. To overcome this difficulty, for any subgroup $G \subset O(N)$ consider the following invariant subspaces:
\begin{equation*}
\begin{aligned}
 X^G:=\setbld{(\eta,\xi)\in X_1\times X_2}{\eta\circ \varphi=\eta,\quad \xi\circ \varphi=\xi, \quad\forall\varphi\in G}, \\    
 Y^G:=\setbld{(\eta,\xi)\in Y_1\times Y_2}{\eta\circ \varphi=\eta,\quad \xi\circ \varphi=\xi, \quad\forall\varphi\in G}.
\end{aligned}
 \end{equation*}
Moreover, consider the restriction $F^G$ of $F$ to $X^G\times P$. Since $F$ is $O(N)$-equivariant by hypothesis, $F^G$ is a well-defined function from $X^G\times P$ into $Y^G$. Moreover, 
\begin{equation*}
    L^G(\rho^\star):=\restr{L(\rho^\star)}{X^G}: X^G\to Y^G
\end{equation*}
is also a well-defined bounded linear mapping. 
It is known that for every $N\ge 2$ and $k^\star\ge 0$ there exists a subgroup $G\subset SO(N)$ such that the subset of $G$-invariant functions in $\cY_{k^\star}$ is a one-dimensional vector space spanned by some spherical harmonic $Y_{k^\star, i^\star}$, which we will simply call $Y^\star$. For instance, if $G:=O(N-1)\times I$, notice that the space of $G$-invariant spherical harmonics (the so-called ``zonal spherical harmonics") of degree $k$ is one-dimensional for all $k$ (we refer to \cite[Appendix]{cava indiana} for a proof of this fact). 

Now, in order to apply Theorem~\ref{Crandall--Rabinowitz theorem} to $F^G$, it will be sufficient to check that the assumptions $(i)$--$(iv)$ are verified. First, $(i)$ holds true by hypothesis. 

Recall that $\pa_X F^G(0,0,\rho^\star)=L^G(\rho^\star):=\restr{L(\rho^\star)}{X^G}$.
We will now show that $\dim \ker L^G = \codim \im L^G =1$, that is condition $(ii)$ in Theorem~\ref{Crandall--Rabinowitz theorem}. 
Let $(\eta,\xi)\in X^G$ be such that $L^G(\rho^\star)[\eta,\xi]=0$. By $(c)$, for all $k\in (\NN\cup\{0\})\setminus\{k^\star\}$ and $i=1,\dots, d_k$, let $\pi_{k,i}$ denote the projection $X^G\to X^G\cap X_{k,i}$. By construction, we have
\begin{equation*}
    \cM(\rho^\star,k)\psi_{k,i}\pi_{k,i}(\eta,\xi)=0.
\end{equation*}
Moreover, $\det\cM(\rho^\star, k)\ne 0$ by $(d)$ and thus $\pi_{k,i}(\eta,\xi)=0$. In other words, we have shown that the projection of any element of $\ker L^G(\rho^\star)$ onto $X^G\cap X_{k,i}$ vanishes for $k\ne k^*$, and thus $\ker L^G(\rho^\star)\subset X_{k^\star,i^\star}$. Again, for any pair $(\eta,\xi)$ in the kernel of $L^G(\rho^\star)$, $(c)$ yields 
\begin{equation*}
    \cM(\rho^\star, k^\star) \psi_{k^\star, i^\star}(\eta,\xi)=0.
\end{equation*}
Recall that, by $(d)$, $\cM(\rho^\star, k^\star)$ is a non-zero, non-invertible $2\times 2$ matrix, thus it has rank $1$. As a result, there exists a pair of real coefficients $(\be,\ga)\ne (0,0)$ such that $\ker L^G(\rho^\star)$ is the one-dimensional vector space spanned by $\left\{ \left( \be Y^\star(\, \cdottone/ R), \ga Y^\star(\cdottone)\right)\right\}$.

In what follows, we will show $(iii)$ and $(iv)$. Notice that it will be enough to consider the case where $L(\rho^\star)+\mu \iota^+$ is a bounded bijection for some $\mu\in \RR$. Indeed, the other case in $(b)$ can be dealt with by simply replacing $F$ with the mapping $(\eta,\xi,\rho)\mapsto F(-\eta,\xi,\rho)$.
Now, let $K: Y^G\to Y^G$ denote the compact operator given by the composition of the inverse of $(L^G(\rho^\star)+\mu\iota^+)$ (which exists by $(b)$) followed by the compact embedding $X^G \hookrightarrow Y^G$. We have 
\begin{equation}\label{Q=}
L^G(\rho^\star)=(\id-\mu K)(L^G(\rho^\star)+\mu\iota^+).
\end{equation}
It follows that 
\begin{equation*}
\im L^G(\rho^\star) = (\id - \mu K) \Big( \underbrace{(L^G(\rho^\star)+\mu\iota^+)(X^G)}_{=Y^G} \Big)=\im (\id-\mu K).
\end{equation*}
Finally, by \cite[Theorem 6.6, (b)]{Brezis}, $\im (\id-\mu K)$ is closed. Moreover, again by \cite[Theorem 6.6, (b), (d)]{Brezis}, we have
\begin{equation*}
\codim \im L^G(\rho^\star) = \codim \im (\id-\mu K) = \dim \ker (\id - \mu K^*)= \dim \ker (\id - \mu K)=1,
\end{equation*}
as claimed. By the above, we can assert that $\im L^G(\rho^\star)$ is a closed subspace of $Y^G$ of codimension $1$, whose orthogonal complement is given by the span of $\left\{ \left( \be Y^\star(R\cdottone), \ga Y^\star(\cdottone)\right)\right\}$ 
Finally, condition $(iv)$ of Theorem~\ref{Crandall--Rabinowitz theorem} follows from $(c)$ and $(e)$. Indeed, 
\begin{equation*}
  \pa_\rho L^G(\rho^\star)\left(\be Y^\star(R\cdottone), \ga Y^\star(\cdottone)\right) = \psi_{k^\star,i^\star}^{-1}\pa_\rho \cM(\rho^\star, k^\star) \begin{pmatrix}
      \be \\ \ga
  \end{pmatrix}
\end{equation*}
is a nonzero element of the span of $ \left\{ \left( \be Y^\star(R\cdottone), \ga Y^\star(\cdottone)\right)\right\}$. 
In other words, the left-hand side in the above is a nonzero element of the orthogonal complement of $\im L^G(\rho^\star)$. This concludes the proof of Lemma~\ref{lem bifurcation}.
\end{proof}
\subsection{The real work}

In what follows, let $D_0$ and $\Om_0$ denote the open balls of $\rn$ centered at the origin with radii $R$ ($0<R<1$) and $1$ respectively. We remark that when 
\begin{equation}\label{T}
    T=T(R):=(1-\sg_c)R^2+\sg_c,
\end{equation}
the overdetermined problem \eqref{new odp in annular config}--\eqref{new odc} admits the following radial solution in $\Om_0\setminus\ol D_0$: 
\begin{equation}\label{u=}
    u(x)=\frac{|x|^2-1}{2}\quad \text{for } R\le|x|\le 1.
\end{equation}
In what follows, we will use a perturbation argument to show the existence of a nontrivial pair of domains $(D, \Om)$ such that the overdetermined problem \eqref{new odp in annular config}--\eqref{new odc} admits a solution.

For some $0<\al<1$, consider the following Banach spaces endowed with their natural norms: $X_1:=C^{2,\al}(\pa D_0)$, $X_2:=C^{2,\al}(\pa\Om_0)$, $Y_1:=C^{1,\al}(\pa D_0)$, $Y_2:=C^{1,\al}(\pa \Om_0)$, and $X:=X_1\times X_2$, $Y:=Y_1\times Y_2$. 
For sufficiently small $(\eta,\xi)\in X$ and $0<\rho<1$, let $\Om_\xi$, $D^\rho_\eta$ denote the bounded domains whose boundaries are given by:
\begin{equation*}
  \pa\Om_\xi:= \setbld{x+\xi(x)\nu(x)}{x\in \pa\Om_0}, 
  \quad \pa D^\rho_\eta:= \setbld{x+(\eta(x)+\rho-R)\nu(x)}{x\in \pa D_0}, 
\end{equation*}
where $\nu(x)=x/|x|$.
Also, let $u_{\eta,\xi, \rho}$ denote the unique solution to the boundary value problem \eqref{new odp in annular config} in the perturbed annular domain $\Om_\xi\setminus \ol D^\rho_\eta$, where $T=T(\rho)$ is defined according to \eqref{T}.

We are interested in how the solution $u_{\eta,\xi, \rho}$ of \eqref{new odp in annular config} in $\Om_\xi\setminus\ol D^\rho_\eta$ changes ``pointwise" with respect to the three parameters $\eta$, $\xi$ and $\rho$. To this end, we will compute its shape derivative. 
The main technical difficulties lie in the following two points: firstly, the functions $u_{\eta,\xi,\rho}$ depend on three parameters, and secondly, each $u_{\eta,\xi,\rho}$ lies in a different function space depending on the choice of $(\eta,\xi,\rho)$.
To overcome these difficulties, we will make use of the following construction. Let 
\begin{equation*}
    E: C^{2,\al}(\pa D_0)\times C^{2,\al}(\pa \Om_0)\times \mathbb{R}\to C^{2,\al}(\rn,\rn)
\end{equation*}
be a bounded linear ``extension operator" that satisfies 
\begin{equation}\label{satisfies}
\restr{E(\eta,\xi,r)}{\pa D_0}=(\eta+r)\nu, \quad 
\restr{E(\eta,\xi,r)}{\pa \Om_0}=\xi\nu.
\end{equation}
Moreover, consider the following pulled-back function: 
\begin{equation}\label{pulled-back function}
    U(\eta,\xi,\rho):= u_{\eta,\xi,\rho}\circ(\id+E(\eta,\xi,\rho-R))\in H^1(\Om_0),\quad \text{for $0<\rho<1$ and small $(\eta,\xi)\in X$}. 
\end{equation}
Then, the (first-order) \emph{shape derivative} of $u_{\eta,\xi,\rho}$ at $(\eta,\xi,\rho)=(0,0,R)$ is defined as 
\begin{equation}\label{shape derivative}
    u'[\eta,\xi,\rho]:=U'(0,0,R)[\eta,\xi,\rho]-\left\langle\gr U(0,0,R), E(\eta,\xi,\rho)\right\rangle,
\end{equation}
where $U'(0,0,R)[\eta,\xi,\rho]$ denotes the Fr\'echet derivative of the pulled-back function $U$ at $(0,0,R)$ in the direction $(\eta,\xi,\rho)$. 
Also, for the sake of simpler notation, we will just write $u'[\eta,\xi]$ instead of $u'[\eta,\xi,0]$.
Finally, notice that the definition given in \eqref{shape derivative} is devised in such a way as to be compatible with a formal application of partial differentiation with respect to $(\eta,\xi,\rho)$ in \eqref{pulled-back function}.
\begin{lemma}\label{u' characterization}
The function $U: X\times (0,1)\to C^{2,\al}(\ol \Om_0\setminus D_0)$ defined in \eqref{pulled-back function} is Fr\'echet differentiable in a neighborhood of $(0,0, R)$. Moreover, for all pairs $(\eta,\xi)\in C^{2,\al}(\pa D_0)\times C^{2,\al}(\pa\Om_0)$, the shape derivative $u'[\eta,\xi]$ at $\rho=R$ is the unique solution to the following boundary value problem. 
\begin{equation}\label{u' eq}
    \begin{cases}
        \De u'=0\quad \tin \Om_0\setminus\ol D_0,\\
        u'=\left(-u_\nu  + \frac{\langle x,\nu\rangle}{\sg_c}\right)\eta=\frac{1-\sg_c}{\sg_c} R \,\eta\quad \ton \pa D_0,\\
         u'=-u_\nu \xi =-\xi \quad \ton \pa\Om_0.
    \end{cases}
\end{equation}
\end{lemma}
\begin{proof}
The Fr\'echet differentiability of the function $U$ in the $C^{2,\al}$-norm (and thus, the shape differentiability of $u_{\eta,\xi,\rho}$) follows from the standard theory of shape differentiability in regular functional spaces (see \cite[Subsection 5.3.6]{HP2005}). Moreover, following \cite[Section 5.6]{HP2005} and the references therein, it is clear that the shape derivative $u'[\eta,\xi]$ is harmonic in the interior of $\Om_0\setminus\ol D_0$ and that the boundary conditions on $\pa (\Om_0\setminus\ol D_0)$ coincide with those obtained by a formal differentiation of the boundary conditions in \eqref{new odp in annular config}. The values of $\restr{u'[\eta,\xi]}{\pa D_0\cup\pa\Om_0}$ can be computed by means of G\^ateaux derivatives, as shown below. 

By combining \eqref{shape derivative}, \eqref{satisfies}, and \eqref{new odp in annular config}, we get the following:
\begin{equation*}
\begin{aligned}
&\text{for }x\in \pa D_0:\quad &u'[\eta,\xi](x)=\restr{\frac{d}{dt}}{t=0} \left(\frac{|x+t\eta(x)\nu(x)|^2-T(R)}{2\sg_c}  \right)-\left\langle\gr u(x), \eta(x)\nu(x)\right\rangle\\
& &= \frac{\eta(x)\langle x,\nu(x)\rangle}{\sg_c}-u_\nu(x)\eta(x)= \frac{1-\sg_c}{\sg_c} R\eta(x);\\
&\text{for }x\in \pa \Om_0:\quad &u'[\eta,\xi](x)=-\left\langle\gr u(x), \xi(x)\nu(x)\right\rangle= -u_\nu(x)\xi(x)=-\xi(x),
\end{aligned}
\end{equation*}
where in the last equalities we have used the fact that, by \eqref{u=}, $u_\nu= |x|$ on $\pa D_0\cup \pa\Om_0$ and $\nu=x/R$ on $\pa D_0$.
\end{proof}

\begin{corollary}\label{cor: explicit computations}
Consider the pair $(\eta,\xi)\in X$ given by the following expression for some coefficients $\be,\ga\in\RR$, $k\in \NN\cup\{0\}$ and $i\in \{1,\dots, d_k\}$ :
 \begin{equation}\label{expansions}
    \eta(R\te)= \be Y_{k,i}(\te),\quad 
    \xi(\te)=\ga Y_{k,i}(\te), \quad \text{for }\te\in\SS^{N-1}.   
 \end{equation}
 Then, the following holds true for all $r\in (R,1)$ and $\te\in \SS^{N-1}$.
\begin{equation}\label{u' linear decomposition}
 u'[\eta,\xi](r\te)= 
 \left\{ (\be A_k+\ga C_k) s_k(r)+ (\be B_k+\ga D_k)t_k(r)\right\}Y_{k,i}(\te), 
\end{equation}
where, for $N\ge 3$ or $k\ge 1$:
\begin{equation}\label{A_k, B_k, C_k, D_k higher}
\begin{aligned}
 &s_k(r):= r^k,\quad  &t_k(r):=r^{2-N-k},\\
    &A_k:= \frac{1-\sg_c}{\sg_c}\frac{R^{N-1+k}}{R^{N-2+2k}-1} ,\quad  &B_k:= \frac{\sg_c-1}{\sg_c} \frac{R^{N-1+k}}{R^{N-2+2k}-1} ,\\ 
    &C_k:=\frac{1}{R^{N-2+2k}-1} ,\quad  &D_k:= \frac{-R^{N-2+2k}}{R^{N-2+2k}-1},
\end{aligned}
\end{equation}
and for $N=2$ and $k=0$:
\begin{equation}\label{A_k, B_k, C_k, D_k N=2 and k=0}
\begin{aligned}
&s_0(r):= 1,\quad  &t_0(r):=\log r,\\
    &A_0:= 0 ,\quad  &B_0:= \frac{1-\sg_c}{\sg_c}\frac{R}{\log R},\\ &C_0:=-1 ,\quad  &D_0:= \frac{1}{\log R}.
\end{aligned}
\end{equation}
\end{corollary}
\begin{proof}
Let us pick arbitrary $k\in\NN\cup\{0\}$ and $i\in\{1,\dots, d_k\}$. We will use the method of separation of variables to find the solution of problem \eqref{u' eq} when the pair $(\eta,\xi)$ is given by \eqref{expansions}.
We will be searching for solutions to \eqref{u' eq} of the form $u'[\eta,\xi]=u'(r,\theta)=f(r)g(\theta)$ (where $r:=\abs{x}$ and $\theta:=x/\abs{x}$ for $x\ne 0$).
Using \eqref{Reilly?} to decompose the Laplacian into its radial and angular components, the equation $\De u'=0$ in $\Om_0\setminus\ol D_0$ can be rewritten as
\begin{equation}
f_{rr}(r)g(\theta)+\frac{N-1}{r} f_r(r)g(\theta)+\frac{1}{r^2}f(r)\De_{\tau}g(\theta)=0 \quad\text{for } R<r<1,\; \theta\in\SS^{N-1}.
\end{equation}
Take $g=Y_{k,i}$.
Under this assumption, we get the following equation for $f$:
\begin{equation}\label{f}
f_{rr}(r)+\frac{N-1}{r}f_r(r)-\frac{\lambda_k}{r^2}f(r)=0 \quad\text{for } R<r<1.
\end{equation}
Since we know that $\la_k=k(k+N-2)$, it can be easily checked that any solution to \eqref{f} consists of a linear combination of the following two independent solutions $s_k$ and $t_k$:
\begin{equation*}
s_k(r):= r^k \text{ for }k\in \NN\cup\{0\},\quad t_k(r):= r^{2-N-k}\text{ for }2-N-k\ne 0, \quad t_0(r):=\log r \text{ for }N=2. 
\end{equation*}
As the solution mapping 
$\RR^2\ni(\be,\ga)\mapsto u'[\eta,\xi]$ is linear, it follows that there exist some real constants $A_k$, $B_k$, $C_k$ and $D_k$ such that \eqref{u' linear decomposition} holds.

Now, with \eqref{expansions} at hands, the boundary conditions in \eqref{u' eq} can be expressed as the following system:
\begin{equation*}
\begin{cases}
A_k s_k(R)+B_k t_k(R)=\frac{1-\sg_c}{\sg_c}R,\\
C_k s_k(R)+D_k t_k(R)=0,\\
A_k s_k(1)+B_k t_k(1)=0,\\
C_k s_k(1)+D_k t_k(1)=-1.
\end{cases}    
\end{equation*}
Finally, by solving it we obtain the desired coefficients in  \eqref{A_k, B_k, C_k, D_k higher}--\eqref{A_k, B_k, C_k, D_k N=2 and k=0}.
\end{proof}

Consider the following mapping $F: X\times \RR \to Y$  
\begin{equation}\label{F}
    \begin{aligned}
        (\eta,\xi, \rho)\mapsto\left( F_1(\eta,\xi, \rho), F_2(\eta,\xi, \rho)       \right),
    \end{aligned}
\end{equation}
where 
\begin{equation*}
\begin{aligned}
&X\times \RR \ni (\eta,\xi, \rho)\mapsto  F_1(\eta,\xi, \rho):= \big\langle\restr{(\id- \gr u_{\eta,\xi, \rho})}{\pa D^\rho_\eta}, \nu^\rho_\eta\big\rangle\circ(\id+(\eta+\rho-R)\nu)\; \in Y_1,\\
&X\times \RR\ni (\eta,\xi, \rho)\mapsto F_2(\eta,\xi, \rho):= \restr{(\pa_{\nu_\xi} u_{\eta,\xi, \rho}- 1)}{\pa \Om_\xi}\circ(\id+\xi\nu)\; \in Y_2,
\end{aligned}
 \end{equation*}
where $\nu_\eta^\rho$ and $\nu_\xi$ denote the outward unit normal vectors to $\pa D_\eta^\rho$ and $\pa\Om_\xi$ respectively. 
Notice that, for all fixed $0<\rho<1)$, the mapping $F$ is well-defined in a neighborhood of $(0,0,\rho)\in X\times \RR$. 
Moreover, by construction, $F(\eta,\xi, \rho)$ vanishes if and only if the solution of \eqref{eq:Dirichlet problem} with respect to the pair $(D^\rho_\eta,\Om_\xi)$ admits two overdetermined level lines $\pa\om_1\subset D^\rho_\eta$ and $\pa\om_2=\pa\Om_0$. In particular, it is easy to verify that, by the definition of $T=T(R)$ in \eqref{T}, $F(0,0, R)=(0,0)$ for all $0<R<1$.

The following lemmas are concerned with the well-definedness and the Fr\'echet differentiability of $F$ in a neighborhood of $(0,0, R)\in X\times \RR$.

\begin{lemma}\label{label F well def}
Let $0<R<1$ and $i=1,2$. Then, $F_i$ is a well-defined mapping from a neighborhood of $(0,0,R)\in X\times \RR$ into $Y_i$.
\end{lemma}
\begin{proof}
    First, notice that, for $0<r<1$ and small enough $(\eta,\xi)\in X$, the sets $D^\rho_\eta$ and $\Om_\xi$ are well defined domains of class $C^{2,\al}$ satisfying $\ol D^\rho_\eta\subset \Om_\xi$. As a result, $\nu^\rho_\eta\in C^{1,\al}(\pa D_\eta^\rho,\rn)$, $\nu_\xi\in C^{1,\al}(\pa \Om_\xi,\rn)$ and $u_{\eta,\xi, \rho}\in C^{2,\al}(\ol \Om_\xi\setminus D^\rho_\eta)$. Then, it follows by composition that $F_1(\eta,\xi,\rho)\in C^{1,\al}(\pa D_0)$ and $F_2(\eta,\xi,\rho)\in C^{1,\al}(\pa D_0)$. 
\end{proof}
The following lemma further shows the Fr\'echet differentiability of $F_1$, $F_2$ and gives an explicit formula for their Fr\'echet derivatives.
\begin{lemma}\label{DF_1, DF_2}
For all $0<R<1$ and $i=1,2$, $F_i$ is a Fr\'echet differentiable mapping from a neighborhood of $(0,0,R)\in X\times \RR$ into $Y_i$, whose Fr\'echet derivatives are given by: 
\begin{equation}\label{first expression for the frechet derivatives}
      \pa_X F_1(0,0,R)[\eta,\xi]= \restr{-\pa_\nu u'[\eta,\xi]}{\pa D_0},\quad
      \pa_X F_2(0,0,R)[\eta,\xi]= \restr{\pa_\nu u'[\eta,\xi]}{\pa \Om_0}+\xi.
\end{equation}
Moreover, under \eqref{expansions}, the expressions in \eqref{first expression for the frechet derivatives} become:
\begin{equation}\label{second expression for the frechet derivatives}
\begin{aligned}
          \pa_X F_1(0,0,R)[\eta,\xi]=  \left\{ \cA_k\be +\cB_k\ga \right\} Y_{k,i}(\,\cdottone /R),\\
      \pa_X F_2(0,0,R)[\eta,\xi]=  \left\{ \cC_k\be+\cD_k\ga \right\} Y_{k,i}(\cdottone),
      \end{aligned}
\end{equation}
where, for $N\ge 3$ or $k\ge 1$:
\begin{equation*}
\begin{aligned}    
    &\cA_k:= \left(\frac{\sg_c-1}{\sg_c}\right)\frac{kR^{N-2+2k}+(k+N-2)}{R^{N-2+2k}-1},  &\cB_k:= \frac{(2-N-2k)R^{k-1}}{R^{N-2+2k}-1},\\&\cC_k:= \left(\frac{\sg_c-1}{\sg_c}\right)\frac{(2-N-2k)R^{N-1+k}}{R^{N-2+2k}-1},
    &\cD_k:=   \frac{(k+N-1)R^{N-2+2k}+(k-1)}{R^{N-2+2k}-1},
\end{aligned}
\end{equation*}
while, for $N=2$ and $k=0$:
\begin{equation*}
    \cA_0:= \left(\frac{\sg_c-1}{\sg_c}\right)\frac{1}{\log R},\quad \cB_0:= -\frac{1}{R\log R}, \quad \cC_0:= \left(\frac{\sg_c-1}{\sg_c}\right)\frac{-1}{\log R},\quad    \cD_0:=  \frac{1+R\log R}{R\log R}.
\end{equation*}
\end{lemma}
\begin{proof}
There are three claims in this lemma. Namely, the Fr\'echet differentiability of the maps $F_1$ and $F_2$, the computation of their Fr\'echet derivatives in \eqref{first expression for the frechet derivatives}, and their explicit formulas under \eqref{expansions}.

First, we will show that 
the maps $F_1$, $F_2$ are Fr\'echet differentiable in a neighborhood of $(0,0,R)\in X\times\RR$.
To this end, notice that $F_1$ and $F_2$ can be rewritten as 
\begin{equation*}
    \begin{aligned}
    & F_1(\eta,\xi,\rho)=\\
     &   \Big\langle \restr{\big\{ (\id+E(\eta,\xi,\rho-R))-\gr u_{\eta,\xi,\rho}\circ(\id+E(\eta,\xi,\rho-R))    \big\}}{\pa D_0} , \restr{\big\{\nu_{\eta,\xi}^\rho \circ (\id+ E(\eta,\xi,\rho-R))\big\}}{\pa D_0}\Big\rangle,\\
      &  F_2(\eta,\xi,\rho)= \Big\langle\restr{\big\{ \gr u_{\eta,\xi,\rho}\circ(\id+E(\eta,\xi,\rho-R))    \big\}}{\pa \Om_0} , \restr{\big\{\nu_{\eta,\xi}^\rho \circ (\id+ E(\eta,\xi,\rho-R))\big\}}{\pa \Om_0}\Big\rangle-1,
    \end{aligned}
\end{equation*}
where $\nu_{\eta,\xi}^\rho$ is the extension of the normals $\nu_{\eta}^\rho$ and $\nu_\xi$ to the whole boundary $\pa(\Om_\xi\setminus\ol D_\eta^\rho)$.
We will show that each ``ingredient" in the expression above is Fr\'echet differentiable in the respective function space. 
First, notice that, by applying the chain rule to \eqref{shape derivative}, we get
\begin{equation}\label{ pulled-back gradient}
\gr u_{\eta,\xi,\rho} \circ(\id + E(\eta,\xi,\rho-R))= \left( \identmatrix+DE(\eta,\xi,\rho-R) \right)^{-T} \gr U(\eta,\xi,\rho),
\end{equation}
where $\gr$ stands for the gradient with respect to the space variable of a real-valued function, $D E(\eta,\xi,\rho-R)$ for the Jacobian matrix of $E(\eta,\xi,\rho-R)$ with respect the space variable, $\identmatrix$ for the identity matrix in $\RR^{N\times N}$, and the superscript $-T$ stands for the inverse transposed matrix. Now, Since $E$ is bounded and linear, the Fr\'echet differentiability of the expression \eqref{ pulled-back gradient} with respect to $(\eta,\xi,\rho)$ at $(0,0,R)$ follows from that of $\gr U(\eta,\xi,\rho)$, which in turn is implied by Lemma~\ref{u' characterization}. 
The last ingredient to be dealt with is the pullback of the perturbed normal. 
Let $\nu(x):=x/|x|$, then it is known (see \cite[Proposition 5.4.14]{HP2005}) that 
\begin{equation}\label{pulled-back normal}
\nu_{\eta,\xi}^\rho\circ(\id+E(\eta,\xi,\rho-R))= \frac{(\identmatrix+DE(\eta,\xi,\rho-R))^{-T}\nu}{|{(\identmatrix+DE(\eta,\xi,\rho-R))^{-T}\nu}|}\quad \ton \pa D_0\cup \pa\Om_0,
\end{equation}
where $|{\cdottone}|$ denotes the Euclidean norm in $\rn$. In particular, the expression in \eqref{pulled-back normal} is Fr\'echet differentiable with respect to $(\eta,\xi,\rho)$, as claimed. With \eqref{ pulled-back gradient} and \eqref{pulled-back normal} at hand, the Fr\'echet differentiability of $F_1$ and $F_2$ readily follows by composition.

Now that we have shown Fr\'echet differentiability, in what follows, we will show the expressions in \eqref{first expression for the frechet derivatives} by computing them as G\^ateaux derivatives and making use of the chain rule. As a key tool in our computations, we will employ the following identity, which is obtained by differentiating \eqref{shape derivative} at $(\eta,\xi,0)$ with respect to the space variable:
\begin{equation}\label{gradient of shape derivative}
    \gr u'[\eta,\xi]= \gr U'(0,0,R)[\eta,\xi] - \underbrace{\gr^2u}_{=\identmatrix} E(\eta,\xi,0) -DE(\eta,\xi,0)^T\gr u. 
\end{equation}
We are now ready to compute $\pa_X F_1(0,0,R)[\eta,\xi]$. For $x\in \pa D_0$, we have:
\begin{equation*}
    \begin{aligned}
        \pa_X F_1(0,0,R)[\eta,\xi](x) = \restr{\frac{d}{dt}}{t=0} F_1(t\eta,t\xi,R)(x)\\
        = \restr{\frac{d}{dt}}{t=0}  \left\langle  x+t\eta(x)\nu(x) -\left( \identmatrix +DE(t\eta,t\xi,0)(x)\right)^{-T} \gr U(t\eta,t\xi,R)(x)\ , \ 
        \nu_{t\eta}^R(x+t\eta(x)\nu(x))\right\rangle\\
        = \left\langle \eta(x)\nu(x)+DE(\eta,\xi,0)^T(x)\gr u(x)- \gr U'(0,0,R)[\eta,\xi](x)\ ,\ \nu(x) \right\rangle 
        + \underbrace{\left\langle x-\gr u(x), -\gr_\tau \eta(x) \right\rangle}_{=0},
    \end{aligned}
\end{equation*}
where we have used \cite[Proposition 5.4.14]{HP2005} in the last line. The exact expression for $\pa_X F_1(0,0,R)[\eta,\xi]$ then follows from \eqref{gradient of shape derivative} with \eqref{satisfies} at hand.

Let us now compute $\pa_X F_2(0,0,R)[\eta,\xi]$. For $x\in \pa \Om_0$, we have:
\begin{equation*}
    \begin{aligned}
        \pa_X F_2(0,0,R)[\eta,\xi](x) = \restr{\frac{d}{dt}}{t=0} F_2(t\eta,t\xi,R)(x)\\
        = \restr{\frac{d}{dt}}{t=0} \left\{\left\langle  \left( \identmatrix +DE(t\eta,t\xi,0)(x) \right)^{-T} \gr U(t\eta,t\xi,R)(x) \ , \ 
        \nu_{t\xi}(x+t\xi(x)\nu(x))\right\rangle -1 \right\}\\
        = \left\langle -DE(\eta,\xi,0)^T(x)\gr u(x)+ \gr U'(0,0,R)[\eta,\xi](x) \ ,\ \nu(x) \right\rangle 
        + \underbrace{\left\langle \gr u(x), -\gr_\tau \xi(x) \right\rangle}_{=0},
    \end{aligned}
\end{equation*}
where, again, we have used \cite[Proposition 5.4.14]{HP2005} in the last line. As before, the exact expression for $\pa_X F_2(0,0,R)[\eta,\xi]$ follows from \eqref{gradient of shape derivative} with \eqref{satisfies} at hand. 

Finally, \eqref{second expression for the frechet derivatives} follows by combining \eqref{first expression for the frechet derivatives} and Corollary~\ref{cor: explicit computations}.
\end{proof}
\subsection{The proof of Theorem~\ref{(1,1) bifurcation}}
Fix $0<R<1$ and let $L:X\to Y$ denote the partial Fr\'echet derivative with respect to the $X$ variable at $(0,0,R)\in X\times\RR$ of the mapping $F$ defined in \eqref{F}.
In what follows, we will show that all conditions $(a)$--$(e)$ of Lemma~\ref{lem bifurcation} are satisfied. First of all, we recall that $(a)$ holds by construction.

We are now ready to show that the mapping $F$ satisfies condition $(b)$ of Lemma~\ref{lem bifurcation}. 
\begin{lemma}\label{proof of (b)}
    Let $\mu<-1$. Then, either $L+\mu\iota^+$ or $L+\mu\iota^-$ is a bounded bijection between $X$ and $Y$.
\end{lemma}
\begin{proof}
   Set $A:=\Om_0\setminus\ol D_0$ and let $n$ denote the outward unit normal vector to $\pa A$ (that is, $n=-\nu$ on $\pa D_0$ and $n=\nu$ on $\pa\Om_0$).  
In order to simplify the notation, we will identify the function space $C^{k,\al}(\pa A)$ with the direct product $C^{k,\al}(\pa 
 D_0)\times C^{k,\al}(\pa \Om_0)$ via $\zeta\mapsto (\restr{\zeta}{\pa D_0}, \restr{\zeta}{\pa\Om_0})$.
 Take a constant $\mu<-1$. We claim that $L+\mu\iota^+: X\to Y$ is a bijection when $0<\sg_c<1$, while $L+\mu\iota^-: X\to Y$ is a bijection when $\sg_c>1$. For the sake of notational simplicity, in what follows, we will just consider the case $0<\sg_c<1$, since the remaining case $\sg_c>1$ is analogous. 
 To this end, let $h, l:\pa D_0\cup \pa\Om_0\to\RR$ denote the following functions:
 
\begin{minipage}{0.5\textwidth}
\begin{equation}\label{h def}
    h:=\begin{cases}
    \frac{1-\sg_c}{\sg_c} R   \quad \ton\pa D_0,\\
    -1 \quad\ton\pa\Om_0,
    \end{cases}
\end{equation}
  \end{minipage}
\begin{minipage}{0.5\textwidth}
  \begin{equation}\label{l def}
     l:=\begin{cases}
    0\quad \ton\pa D_0,\\
    1 \quad\ton\pa\Om_0.
    \end{cases}
\end{equation}
  \end{minipage}

\noindent Then, \eqref{u' eq} can be rewritten as 
   \begin{equation*}
       \begin{cases}
           \De u'=0\quad \tin A,\\
           u'=h\zeta\quad\ton \pa A,
       \end{cases}
   \end{equation*}
   where $h$ is the function defined in \eqref{h def} and $\zeta\in C^{2,\al}(\pa A)$ is the function identified with the pair $(\eta,\xi)\in X$. Also, the standard Schauder boundary estimates \cite[Chapter 6]{GT} combined the fact that $h\ne 0$ on $\pa A$ imply that the mapping $\zeta\mapsto u'[\zeta]$ is a bounded bijection between $C^{2,\al}(\pa A)$ and the Banach space $W:=\setbld{w\in C^{2,\al}(\ol A)}{\De w=0\quad \tin 
 A}$. Recall that, by Lemma~\ref{DF_1, DF_2}, we can write 
\begin{equation}\label{a different way of writing L+mu iota}
    \left(L+\mu\iota^+\right)\zeta= \pa_n u'[\zeta]+\frac{l+ \mu}{h}u'[\zeta],
\end{equation}
First of all, we will show the invertibility of $L+\mu\iota^+$, as a mapping from $C^{2,\al}(\pa A)\to C^{1,\al}(\pa A)$ given by \eqref{a different way of writing L+mu iota}. In other words, for all $f\in C^{1,\al}(\pa A)$ we will find a unique $u'[\zeta]\in W$ (and, thus, a unique $\zeta\in C^{2,\al}(\pa A)$) such that the right-hand side of \eqref{a different way of writing L+mu iota} is equal to $f$. 
To this end consider the following bilinear form:
\begin{equation}\label{B def}
    \cB(w,\phi):=\int_A \langle \gr w, \gr\phi\rangle+\int_{\pa A}\frac{l+\mu}{h} w\phi.
\end{equation}
By definition, $\cB$ is clearly a continuous bilinear form on $H^1(A)\times H^1(A)$. Since $\sg_c<1$ by hypothesis, coercivity now follows from \eqref{h def}--\eqref{l def} and the choice of $\mu$.
Fix now a function $f\in C^{1,\al}(\pa A)$. The Lax--Milgram theorem ensures the existence of a unique function $w\in H^1(A)$ such that, for all $\phi\in H^1(A)$:
\begin{equation*}
    \int_A \langle\gr w, \gr \phi\rangle+\int_{\pa A} \frac{l+\mu}{h}w\phi=\int_{\pa A} f\phi.
\end{equation*}
Notice that the above is nothing but the weak form of 
\begin{equation}\label{w eq}
    \begin{cases}
        \De w=0\quad \tin A,\\
        w_n+\frac{l+\mu}{h}w=f\quad \ton \pa A. 
    \end{cases}
\end{equation}
Set now $\zeta:=\restr{w}{\pa A}/h$. By \eqref{u' eq} we have $w=u'[\zeta]$. Moreover, \eqref{w eq} yields 
\begin{equation*}
    \left(L+\mu\iota\right)\zeta= \restr{\pa_n u'[\zeta]}{\pa A}+l\zeta+\mu\zeta = f. 
\end{equation*}
In other words, for all $f\in C^{1,\al}(\pa A)$, there exists a unique function $\zeta\in L^2(\pa A)$ such that $f$ is equal to the left-hand side of \eqref{a different way of writing L+mu iota}. Now, in order to conclude the proof, it suffices to show that $\zeta$ actually belongs to $C^{2,\al}(\pa A)$, as claimed. To this end, notice that, one can inductively bootstrap the boundary regularity of $w$ (and, thus, that of $\zeta$) in a classical way by means of the standard elliptic regularity estimates \cite[Chapter 8]{GT} and the Schauder interior and boundary estimates \cite[Chapter 6]{GT} (see for example the argument in the proof of \cite[Proposition 5.2]{kamburov sciaraffia} after $(5.7)$). We obtain that $w\in C^{2,\al}(\ol A)$. As a result, $\zeta\in C^{2,\al}(\pa A)$, as claimed.

This concludes the proof of the invertibility of the mapping $L+\mu\iota^+: X\to Y$ when $0<\sg_c<1$. The invertibility of the map $L+\mu\iota^-$ in the case $\sg_c>1$ can be shown in an analogous way by suitably modifying the integral coefficient in the second term of the bilinear form $\cB$ in \eqref{B def}.
\end{proof}

Condition $(c)$ of Lemma~\ref{lem bifurcation} is also verified, the matrix-valued function
\begin{equation*}
    \cM(R,k):= \begin{pmatrix}
        \cA_k\quad \cB_k\\
        \cC_k\quad \cD_k
    \end{pmatrix}
\end{equation*}
being defined according to Lemma~\ref{DF_1, DF_2}.

Condition $(d)$ of Lemma~\ref{lem bifurcation} follows by combining the following two lemmas.

\begin{lemma}\label{lem exists R^star}
For all $0<R<1$, $\det \cM(R,0)\ne 0$ and $\det \cM(R,1)\ne 0$. Moreover, for all integer $k\ge 2$, there exists a unique $R^\star=R^\star(k)$ in the interval $(0,1)$ such that $\det\cM(R^\star(k),k)=0$.    
\end{lemma}
\begin{proof}
The first claim readily follows by a direct computation. Indeed, one has 
\begin{equation*}
    \det \cM(R,0)=\begin{cases}
        \frac{\sg_c-1}{\sg_c} (N-2)\frac{R^{2-N}}{1-R^{2-N}}\ne 0 \quad \text{for }N\ge3,\\
     \frac{\sg_c-1}{\sg_c} \frac{R}{\log R}\ne 0 \quad \text{for }N=2
    \end{cases} 
\end{equation*}
and 
\begin{equation*}
    \det \cM(R,1)= \frac{\sg_c-1}{\sg_c} \frac{N R^N}{(R^N-1)^2}(R^N-1)\ne 0,
\end{equation*}
whence, for $R\in (0,1)$ and $k\in\{0,1\}$, the quantity $\det\cM(R,k)$ does not vanish.

In order to show the second claim, first notice that, for any integer $k\ge 2$, 
\begin{equation*}
    \det\cM(R,k)=\frac{\sg_c-1}{\sg_c} \frac{g(R,k)}{\left( R^{N-2+2k}-1\right)^2},
\end{equation*}
where 
\begin{equation*}
   g(R,k):= (kN+k^2-k) \ R^{2N-4+4k} + (-2kN-2k^2+N+4k-2) \ R^{N-2+2k} + (kN+k^2-N-3k+2). 
\end{equation*}
Since $\sg_c\ne 1$, $k\ge2$, and $0<R<1$, it follows that $\det\cM(R,k)=0$ holds if and only if $g(R,k)=0$. Now, if we regard $g(R,k)=0$ as a quadratic equation in $R^{N-2+2k}$, its two solutions are
\begin{equation*}
\frac{2kN+2k^2-N-4k+2\pm \left(N+2k-2\right)}{2(kN+k^2-k)}=\begin{cases}
        1, \\
        1-\frac{N-2+k}{kN+k^2-k}\in (0,1).
    \end{cases}
\end{equation*}
After further simplifications, we obtain that, for given $k\ge 2$, the equation $\det\cM(\cdottone,k)=0$ has a unique solution $R^\star$ in the interval $(0,1)$, given by
\begin{equation}\label{R^star(k)}
    R^\star=R^\star(k)=\left(1-\frac{N-2+k}{kN+k^2-k}\right)^{1/(N-2+2k)}.
\end{equation}
\end{proof}

\begin{lemma}
    If $\det\cM(R,j)=\det\cM(R,k)=0$ for some $j,k\in\NN$, then $j=k$. Furthermore, $\cM(R,k)\ne 0$ for all $k\in \NN\cup\{0\}$ and $R\in (0,1)$.
\end{lemma}
\begin{proof}
Let us consider the first claim of the lemma. First, notice that, if $\det\cM(R,j)$ and $\det\cM(R,k)$ vanish, then both $j$ and $k$ must be greater than or equal to $2$ by Lemma \ref{lem exists R^star}. The claim then follows since the mapping $k\mapsto R^\star(k)$ given by \eqref{R^star(k)} is strictly monotone increasing in $k$ for $k\ge2$. To see this, notice that, by \eqref{R^star(k)}, $R^\star(k)$ has the form 
\begin{equation*}
    R^\star(k)=a(k)^{b(k)},
\end{equation*}
where $k\mapsto a(k)$ is a strictly increasing function with values in $(0,1)$, and $k\mapsto b(k)$ is a strictly decreasing positive function.

The second claim also readily follows as, in particular, $\cB_k\ne 0$ for $k\in \NN\cup\{0\}$ and $0<R<1$.
\end{proof}

Finally, the following lemma takes care of condition $(e)$ in the case $k^\star\ge 2$, $R^\star=R^\star(k^\star)$.
\begin{lemma}\label{lem det M_R never vanishes}
For all $0<R<1$ and $k\ge 2$, $\det\pa_R\cM(R,k)\ne 0$.    
\end{lemma}
\begin{proof}
The result follows from elementary computations. Indeed, one can check that for $k\ge 2$ we have:
\begin{equation*}\label{det par cm}
 \det\pa_R \cM(R,k)= \frac{1-\sg_c}{\sg_c}
     \frac{R^{2k+N-4}}{(R^{N-2+2k}-1)^2}
     {\left(N + 2 \, k - 2\right)}^{2} {\left(N + k - 1\right)}{\left(k - 1\right)}.
\end{equation*}
Thus, the expression above never vanishes for $\sg_c\ne 1$, $0<R<1$, $N\ge 2$ and $k\ge2$. 
\end{proof}

We can now prove Theorem~\ref{(1,1) bifurcation}.
\begin{proof}[Proof of Theorem~\ref{(1,1) bifurcation}]
Take some $k\in\NN$ with $k\ge 2$ and let $(D_0,\Om_0)$ denote the open balls centered at the origin with radii $R=R^\star(k)$ (defined as in \eqref{R^star(k)}) and $1$ respectively. Since we have shown that all conditions $(a)-(e)$ are met, we can finally apply Lemma~\ref{lem bifurcation} to show the existence of a spherical harmonic $Y^\star$ of degree $k$, a pair of real numbers $(\be,\ga)\ne (0,0)$, and a nontrivial branch  
\begin{equation*}
   (-\varepsilon,\varepsilon)\ni t\mapsto \left(\eta(t), \xi(t), \rho(t) \right)\in X \times  \RR 
\end{equation*}
such that, $\rho(0)=R$, $\eta(0)=0$, $\xi(0)=0$,  $\eta'(0)=\be Y^{\star}(\,{\cdottone}/{R})$, $\xi'(0)=\ga Y^{\star}(\cdottone)$ and $F(\eta(t), \xi(t), \rho(t))=0$ for all $t\in (-\ve,\ve)$. In particular, for $|t|$ small enough, the pair $(D_{\eta(t)}^{\rho(t)},\Om_{\xi(t)})$ satisfies some overdetermination of type $(1,1)^\star$. Moreover, since $(\be,\ga)\ne 0$, for $|t|$ small enough, either $D_{\eta(t)}^{\rho(t)}$ or $\Om_{\xi(t)}$ is not a ball. Actually, one can show that neither are balls. Indeed, if this were the case, one would get a contradiction with either Theorem~\ref{(1,1) D=ball symmetry} (if $D_{\eta(t)}^{\rho(t)}$ were a ball and $\Om_{\xi(t)}$ were not) or \cite[Theorem 5.1]{Sak bessatsu} (if $\Om_{\xi(t)}$ were a ball and $D_{\eta(t)}^{\rho(t)}$ were not). Thus, the claim of Theorem~\ref{(1,1) bifurcation} readily follows in light of Lemma~\ref{star implies nonstar for asymmetry}. 
\end{proof}

\section{Counterexamples for overdetermination of type $(1,0)$}\label{sec:(a,0)}

In this section, we will give a proof of Theorem~\ref{(1,0) asymmetry} via the Cauchy--Kovalevskaya theorem.

Let $D_0:=\setbld{x\in\rn}{|x|<R}$ for some $R>0$. In what follows, we will construct a bounded domain $\Om\supset\ol D_0$ such that $(D_0,\Om)$ are not concentric balls but the solution $u$ to \eqref{eq:Dirichlet problem} in $(D_0,\Om)$ is radial in $D_0$ (but not with respect to the center of $D_0$). This will yield a counterexample to radial symmetry in the presence of overdetermination of type $(a,0)$ for all $a\in \NN\cup\{0\}$.

For small $\ve>0$ consider the following functions:
\begin{equation}\label{f_ve g_ve}
f_\ve(x):=\frac{|x-\ve e_1|^2}{2\sg_c}, \quad g_\ve(x):= \langle x-\ve e_1, \nu\rangle, \quad \text{for }x\in \pa D_0.    
\end{equation}
In \cite{walter}, W. Walter gave an alternative proof of the Cauchy--Kovalevskaya theorem using the Banach fixed-point theorem. As a result, he showed that the solution matching the Cauchy data on a non-characteristic surface is not only uniquely determined by the data (including the equation) of the problem but also continuously dependent on them. If we localize \cite[Theorem 2]{walter} and apply it to our setting, we get the existence of positive constants $R_1$, $R_2$ (with $R_1<R<R_2$), $\ve_0>0$, and of a unique continuous mapping 
\begin{equation}\label{u_cdottone}
    u_\cdottone: [0,\ve_0]\to C^\om\left(\ol B_{R_2}\setminus  B_{R_1}, \RR\right)
\end{equation}
such that, for all $\ve\in [0,\ve_0]$, the real-analytic function $u_\ve$ is the unique solution to the following problem:
\begin{equation}\label{Cauchy problem in a strip}
   \begin{cases}
       \De u_\ve=N, \quad \tin B_{R_2}\setminus \ol B_{R_1},\\
       u_\ve=f_\ve\quad \ton \pa D_0,\\
       \pa_\nu u_\ve =g_\ve \quad \pa D_0.
   \end{cases} 
\end{equation}
We remark that the theorem in \cite{walter} shows continuous dependency in the $C^0$-norm. Nevertheless, continuous dependency in the $C^1$-norm holds as well, as shown in the following lemma.
\begin{lemma}[Continuous dependency in the $C^1$-norm]\label{C^1 dependency}
Let $u_\cdottone$ be the mapping defined by \eqref{u_cdottone}--\eqref{Cauchy problem in a strip}. Then, the following mapping is continuous in the $C^0$-norm: 
\begin{equation*}
    \gr u_\cdottone: [0,\ve_0]\to C^\om\left(\ol B_{R_2}\setminus  B_{R_1}, \rn\right), 
\end{equation*}
\end{lemma}
\begin{proof}
The claim follows by applying once again the Cauchy--Kovalevskaya theorem, \cite[Theorem 2]{walter}, to the partial derivatives of $u_\ve$. To this end, it will be enough to show that, for $i=1,\dots, N$, the Cauchy data satisfied by $\pa_{x_i}u_\ve$ on $\pa D_0$ depend continuously on the parameter $\ve$ in the $C^0$-norm. For arbitrary $\ve\in [0,\ve_0]$, consider the unique solution $u_\ve$ to \eqref{Cauchy problem in a strip}.
Let us first study the Dirichlet data satisfied by $\gr u_\ve$ on $\pa D_0$. 
Item $(i)$ of Lemma~\ref{lem: jump} yields 
\begin{equation*}
{\gr u_\ve}= \gr_\tau u_\ve + \pa_\nu u_\ve\, \nu = \gr_\tau f_\ve + g_\ve\, \nu \quad \text{on }\pa D.
\end{equation*}
Now, by recalling the definitions of $f_\ve$ and $g_\ve$ in \eqref{f_ve g_ve}, the above implies that, for all $i=1,\dots, N$, the Dirichlet data on $\pa D_0$ of $\pa_{x_i}u_\ve$ depends continuously on $\ve$ in the $C^0$-norm. 

Let us now consider the Neumann data satisfied by $\gr u_\ve$ on $\pa D_0$. 
Combining $(ii)$ and $(iii)$ of Lemma~\ref{lem: jump} yields 
\begin{equation*}
\begin{aligned}
\pa_\nu \gr u_\ve = \gr^2 u_\ve \, \nu 
= N- \De_\tau u_\ve - (N-1)H \pa_\nu u_\ve + \gr_\tau \pa_\nu u_\ve - D_\tau \nu \gr_\tau u_\ve \\ = 
N-\De_\tau f_\ve - \frac{N-1}{R} g_\ve + \gr_\tau g_\ve - D_\tau \nu \gr_\tau f_\ve 
\quad  \text{on }\pa D_0.
\end{aligned}
\end{equation*}
As before, we find that, for all $i=1,\dots, N$, the Neumann data on $\pa D_0$ of $\pa_{x_i}u_\ve$ also depends continuously on $\ve$ in the $C^0$-norm. This concludes the proof.
\end{proof}
We are now in a position to prove Theorem~\ref{(1,0) asymmetry}.
\begin{proof}[Proof of Theorem~\ref{(1,0) asymmetry}]
First, notice that, by construction, 
\begin{equation*}
    u_0(x)=|x|^2/2- R^2/2+R^2/(2\sigma_c) \quad \text{for } x\in B_{R_2}\setminus \ol D_0. 
\end{equation*}
As a result, we have the following:
\begin{equation*}
\begin{aligned}
   \left\langle \gr u_0(x), \frac{x}{|x|}\right\rangle=|x|\ge R>0 \quad \text{for }x\in \ol B_{R_2}\setminus D_0,\\
    \max_{\pa D_0} u_0 = \frac{R^2}{2\sg_c} < \frac{R_2^2-R^2}{2}+\frac{R^2}{2\sg_c}=\min_{\pa B_{R_2}}u_0.
\end{aligned}
\end{equation*}
Thus, by Lemma~\ref{C^1 dependency}, we can find some $\ve\in (0,\ve_0)$ such that the following both hold true:
\begin{align}
\left\langle\gr u_\ve(x), \frac{x}{|x|}\right\rangle  >\frac{R}{2} & >0\quad \text{for }x\in \ol B_{R_2}\setminus D_0, \label{star} \\
 \displaystyle\max_{\pa D_0}u_\ve & < \min_{\pa B_{R_2}} u_\ve. \label{double star}
\end{align}
Take now some constant $\displaystyle \ga\in \Big( \max_{\pa D_0} u_\ve, \min_{\pa B_{R_2}} u_\ve \Big)$.

We claim that for all $\te\in \SS^{N-1}$ there exists a unique radius $r(\te)\in (R,R_2)$ such that $u_\ve\left(r(\te)\te\right)=\ga$.
In order to show that, fix an element $\te\in \SS^{N-1}$ and consider the function
\begin{equation}\label{a continuous function}
    (R,R_2)\ni r\mapsto u_\ve(r\te) \in \RR.
\end{equation}
This function is continuous by construction, and monotone because of \eqref{star}. Moreover,
\begin{equation*}
    u_\ve(R\te)\le \max_{\pa D_0} u_\ve <\ga < \min_{\pa B_{R_2}} u_\ve \le u_\ve (R_2 \te).
\end{equation*}
The claim now follows from the intermediate value theorem. 

Consider now the level set 
\begin{equation}\label{u_ve=ga}
 \setbld{x\in \ol B_{R_2}\setminus D_0 }{u_\ve(x)=\ga}.   
\end{equation}
Since, by \eqref{star}, the gradient of $u_\ve$ does not vanish in $\ol B_{R_2}\setminus D_0$ and thus, by the implicit function theorem, the level set in \eqref{u_ve=ga} can be locally written as the graph of an analytic function. In other words, the level set in \eqref{u_ve=ga} is an analytic hypersurface embedded in $\rn$. Thus, we can consider the bounded domain $\Om$ enclosed by it, that is,
\begin{equation*}
    \Om:= \setbld{r\te}{\te\in \SS^{N-1}, \quad 0\le r < r(\te)}.
\end{equation*}

The pair $(D_0,\Om)$ then yields the desired counterexample. Indeed, one can check that the function 
\begin{equation*}
    u(x):=\begin{cases}
        f_\ve(x)-\ga\quad \text{for }x\in D_0,\\
        u_\ve(x)-\ga\quad \text{for }x\in \Om\setminus D_0
    \end{cases}
\end{equation*}
solves \eqref{eq:Dirichlet problem} and is radial with respect to the point $\ve e_1\in D_0$. Then, in particular, any sphere centered at $\ve e_1$ and small enough to fit inside $D_0$ is an overdetermined level set for $u$. Nevertheless, $(D_0,\Om)$ are not concentric spheres. Indeed, if that were the case, by the unique solvability of \eqref{eq:Dirichlet problem}, the function $u$ would be radial with respect to the center of $D_0$ (the origin) and not $\ve e_1$.  
\end{proof}

\section*{Acknowledgements}
The first author is partially supported by JSPS KAKENHI Grant Number JP22K13935, JP21KK0044, and JP23H04459.

The second author is supported by the Australian Research Council (ARC) Discovery Early Career Researcher Award (DECRA) DE230100954 “Partial Differential Equations: geometric aspects and applications” and the 2023 J G Russell Award from the Australian Academy of Science, and is member of the Australian
Mathematical Society (AustMS) and the Gruppo Nazionale Analisi Matematica Probabilit\`a e Applicazioni (GNAMPA) of the Istituto Nazionale di Alta Matematica (INdAM).

Part of this work was conducted while both authors were visiting the University of Florence in June 2023. The authors wish to thank Paolo Salani for sharing his office with them. 
%
%

\begin{small}

\end{small}


\noindent
\textsc{Lorenzo Cavallina:}\\
\noindent
\textsc{
Mathematical Institute, Tohoku University, Aoba-ku, 
Sendai 980-8578, Japan}\\
\noindent
{\em Electronic mail address:}
cavallina.lorenzo.e6@tohoku.ac.jp

\bigskip

\noindent
\textsc{Giorgio Poggesi:}\\
\noindent
\textsc{
Discipline of Mathematical Sciences, The University of Adelaide, Adelaide SA 5005, Australia}
\\
\noindent
{\em Electronic mail address:}
giorgio.poggesi@adelaide.edu.au
\\
\noindent
\textsc{and}
\\
\noindent
\textsc{Department of Mathematics and Statistics, The University of Western Australia, 35 Stirling Highway, Crawley, Perth, WA 6009, Australia} \\
\noindent
{\em Electronic mail address:}
giorgio.poggesi@uwa.edu.au


\begin{thebibliography}{99}

\bibitem
{ABMMZ} R. Alvarado, D. Brigham, V. Maz'ya, M. Mitrea, E. Ziad\'{e}, {\em On the regularity of domains satisfying a uniform hour-glass condition and a sharp version of the {H}opf-{O}leinik boundary point principle}, Problems in mathematical analysis. no. 57. J. Math. Sci. (N.Y.), 176 (2011), no. 3, 281--360.

\bibitem{AthaStra}
\textsc{C. Athanasiadis, I.G. Stratis}, \emph{On some elliptic transmission problems},
Annales Polonici Mathematici 63 (1996), no. 2, 137--154.

\bibitem
{Brezis}
\textsc{H. Brezis}, Functional Analysis, Sobolev spaces and Partial Differential Equations, Springer (2011).

\bibitem{cava indiana} \textsc{L. Cavallina}, {\em Local analysis of a two phase free boundary problem concerning mean curvature}. Indiana University Mathematics Journal, 71 (2022), no. 4, 1411--1435.

\bibitem{cava symm asymm}\textsc{L.~Cavallina}, \emph{Symmetry and asymmetry in a multi-phase overdetermined problem}, Interfaces and
Free Boundaries 26 (2024), no. 3, 473--488. https://doi.org/10.4171/ifb/512


\bibitem{bessatsata}\textsc{L.~Cavallina, G.~Poggesi}, {\em Elliptic and parabolic overdetermined problems in multi-phase settings
}, 	arXiv:2504.18808

\bibitem
{CY2020}\textsc{L.~Cavallina, T.~Yachimura}, {\em Symmetry breaking solutions for a two-phase overdetermined problem of Serrin-type}, Current trends in analysis, its applications and computation, 433--441.
Trends Math. Res. Perspect.
Birkhäuser/Springer, Cham, (2022).

\bibitem{CR}
\textsc{M.G. Crandall, P.H. Rabinowitz},
{\em Bifurcation from simple eigenvalues}. Journal of Functional Analysis, 8 (1971), no. 2, 321--340.

\bibitem
{DPV} \textsc{S.~Dipierro, G.~Poggesi, E.~Valdinoci}, \emph{A Serrin-type problem with partial knowledge of the domain}, Nonlinear Anal., 208 (2021), 112330, 44 pp.

\bibitem
{GT} \textsc{D.~Gilbarg, N.S.~Trudinger}, {Elliptic Partial Differential Equation of Second Order. Reprint of the 1998 edition
Classics Math.
Springer-Verlag, Berlin }(2001).

\bibitem
{HP2005} \textsc{A.~Henrot, M.~Pierre}, {Shape variation and optimization
(a geometrical analysis)}. EMS Tracts in Mathematics,
Vol.28, European Mathematical Society (EMS),
Z\"urich, (2018).

\bibitem
{kamburov sciaraffia}
\textsc{N. Kamburov, L. Sciaraffia},
{\em Nontrivial solutions to Serrin's problem in annular domains}, Annales de l'Institut Henri Poincar\'e C, Analyse non lin\'eaire 38 (2021), no. 1, 1--22.

\bibitem
{LiVogelius} \textsc{Y.Y. Li, M. Vogelius}, \emph{Gradient estimates for solutions to divergence form elliptic equations with discontinuous coefficients}, Arch. Ration. Mech. Anal., 153 (2000), no. 2, 91--151.

\bibitem
{Kinderlehrer Nirenberg}
\textsc{D. Kinderlehrer, L. Nirenberg}, \emph{Regularity in free boundary problems}, Annali della Scuola Normale Superiore di Pisa - Classe di Scienze, S\'erie 4, Tome 4 (1977) no. 2, 373--391.

\bibitem
{MP2} \textsc{R.~Magnanini, G.~Poggesi}, \emph{Serrin's problem and Alexandrov's Soap Bubble Theorem: enhanced stability via integral identities}, Indiana Univ. Math. J., 69 (2020), no. 4, 1181--1205.

\bibitem{NT2018} \textsc{C.~Nitsch, C.~Trombetti}, {\em The classical overdetermined Serrin problem}. Complex Variables and Elliptic Equations, 63:7-8 (2018), 1107--1122.

\bibitem{Payne-Schaefer}
\textsc{L.E. Payne, P.W. Schaefer}, \emph{Duality theorems in some overdetermined boundary value problems}, Math. Methods Appl. Sci. 11 (1989), no. 6, 805--819.

\bibitem
{PogTesi} \textsc{G.~Poggesi}, \emph{The Soap Bubble Theorem and Serrin's problem: quantitative symmetry}, PhD Thesis, Universit\`a di Firenze, defended on February 2019, preprint arxiv:1902.08584.

\bibitem{Reilly}
\textsc{R.C.~Reilly}, \emph{Mean curvature, the Laplacian, and soap bubbles}, Amer. Math. Monthly 89 (1982), no.3, 180--188, 197--198.

\bibitem{Sak bessatsu}
\textsc{S.~Sakaguchi},
{\em Two-phase heat conductors with a stationary isothermic surface and their related elliptic overdetermined problems}, RIMS K\^oky\^uroku Bessatsu, B80 (2020), 113--132.

\bibitem
{Se1971} \textsc{J.~Serrin}, {\em A symmetry problem in potential theory}. Arch. Rat. Mech. Anal., 43 (1971): 304--318.

\bibitem{Sirakov} 
\textsc{B. Sirakov}, {\em Symmetry for exterior elliptic problems and two conjectures in potential theory}, Ann. Inst. H. Poincar\'e C Anal. Non Lin\'eaire 18 (2001), no.2, 135--156.

\bibitem
{Vo} \textsc{A.~L.~Vogel}, {\em Symmetry and regularity for general regions having a solution to certain overdetermined boundary value problems}, Atti Sem. Mat. Fis. Univ. Modena 40 (1992), no. 2, 443--484.

\bibitem{walter}
\textsc{W.~Walter}, {\em An elementary proof of the Cauchy-Kowalevsky theorem}, Am. Math. Mon. 92 (1985) no.2, 115--126. 

\bibitem{Weinberger}
\textsc{H.F. Weinberger}, \emph{Remark on the preceding paper of Serrin}, Arch. Ration. Mech. Anal. 43 (1971), 319--320.

\bibitem
{XB2013} \textsc{J. Xiong, J. Bao}, {\em Sharp regularity for elliptic systems associated with transmission problems}, Potential Anal., 39 (2013), no. 2, 169--194.

\bibitem
{Zhuge} \textsc{J.~Zhuge}, {\em Regularity of a transmission problem and periodic homogenization}, J. Math. Pures Appl. 153 (2021), 213--247.

\end{thebibliography}
\end{document}